\documentclass{itor}
\usepackage{natbib}%

\usepackage[T1]{fontenc}

\usepackage{graphicx}
\usepackage{amsmath, amssymb}
\usepackage{amsthm}
\usepackage{multirow}
\usepackage{graphicx}
\usepackage{amsfonts}
\usepackage{subcaption}
\usepackage{hyperref}
\usepackage{appendix}

\usepackage{color}	
\usepackage{booktabs}

\usepackage{enumerate}
\usepackage{enumitem}
\usepackage{mathtools}
\usepackage[linesnumbered]{algorithm2e}
\usepackage[section]{placeins}
\SetKwRepeat{Do}{do}{while}
\let\oldnl\nl
\newcommand\nonl{%
  \renewcommand{\nl}{\let\nl\oldnl}}

\makeatletter
\DeclareRobustCommand{\change}{%
  \@bsphack
  \leavevmode
  \color{red}%
  \@esphack
}
\DeclareRobustCommand{\stopchange}{%
  \@bsphack
  \normalcolor
  \@esphack
}
\makeatother

\numberwithin{table}{section}

\newtheorem{theorem}{\textbf{Theorem}}
\newtheorem{observation}[theorem]{\textbf{Observation}}
\newtheorem{definition}{\textbf{Definition}}
\newtheorem{problem}{\textbf{Problem}}
\newtheorem{proposition}[theorem]{\textbf{Proposition}}
\newtheorem{remark}[theorem]{\textbf{Remark}}
\newtheorem{lemma}[theorem]{\textbf{Lemma}}

\jname{International Transactions in Operational Research}
\jvol{XX}
\jyear{20XX}
\doi{xx.xxxx/itor.xxxxx}

\begin{document}

\title{Strong bounds and exact solutions to the minimum broadcast time problem}

\author[Ivanova et al.]{Marika Ivanova\affmark{a,$\ast$}, Dag Haugland\affmark{b} and B{\aa}rd Hennning Tvedt\affmark{c}}
\affil{\affmark{a}Department of Theoretical Computer Science and Mathematical Logic, Charles University, Prague, Czech Republic}
\affil{\affmark{b}Department of Informatics, University of Bergen, Norway}
\affil{\affmark{c}Webstep, Bergen, Norway}

\email{ivanova@ktiml.mff.cuni.cz [M.\ Ivanova]; \\Dag.Haugland@uib.no [D.\ Haugland]; bard.tvedt@webstep.no [B.H.\ Tvedt]}
\thanks{\affmark{$\ast$}Author to whom all correspondence should be addressed (e-mail: ivanova@ktiml.mff.cuni.cz).}

\historydate{Received DD MMMM YYYY; received in revised form DD MMMM YYYY; accepted DD MMMM YYYY}

\begin{abstract}
Given a graph and a subset of its nodes, referred to as source nodes, the minimum broadcast problem asks for the minimum number of steps in which a signal can be transmitted from the sources to all other nodes in the graph.
In each step, the sources and the nodes that already have received the signal can forward it to at most one of their neighbour nodes.
The problem has previously been proved to be NP-hard. 
In the current work, we develop a compact integer programming model for the problem.
We also devise procedures for computing lower bounds on the minimum number of steps required, along with methods for constructing near-optimal solutions.
Computational experiments demonstrate that in a wide range of instances, \change in particular instances \stopchange with sufficiently dense graphs, the lower and upper bounds under study collapse.
In instances where this is not the case, the integer programming model proves strong capabilities in closing the remaining gap,
\change and proves to be considerably more efficient than previously studied models.\stopchange
\end{abstract}

\keywords{Broadcasting; Integer Programming; Bounds; Computational Experiments}
\maketitle

\section{Introduction}
\label{intro}
\footnote{The research was partially supported by OP RDE project No. CZ.02.2.69/0.0/0.0/18\_053/0016976.}
Fast and efficient distribution of information gives rise to many optimisation problems of growing interest.
Information dissemination processes studied in the mathematical and algorithmic literature \citep{fraigniaud94, harutyunyan13, hedetniemi88, hromkovic96}
often fall into one of the categories \emph{gossiping} or \emph{broadcasting}.
When each network node controls its own, unique piece of information, and all pieces are to be disseminated to all nodes,
the process is called gossiping \citep{bermond98, bermond95}.
Dissemination of the information controlled by one particular source node to all network nodes is referred to as broadcasting \citep{mcgarvey16, ravi94},
and multicasting \citep{barnoy00} if a subset of the network nodes are information targets.
If the information is to be stored at the source, and assembled by pieces stored at all other nodes, then the information flows in the reverse of the broadcasting direction,
and the dissemination process is \emph{accumulation}.
Broadcasting and accumulation can both be generalised to processes where only a subset of the nodes need to receive/disseminate information,
while the remaining nodes are available as transit units that pass the information on to neighbouring nodes.

Information dissemination follows a certain \emph{communication model}.
In the \emph{whispering} model, each node sends/receives information to/from at most one other node in its vicinity at \change a \stopchange time.
The \emph{shouting} model corresponds to the case where nodes communicate with all their neighbour nodes simultaneously.
Generalising whispering and shouting, the communication can also be constrained to neighbour subsets of given cardinality.

In the current work, a problem in the domain of broadcasting is studied.
The \emph{minimum broadcast time} (MBT) problem is identified by a graph and a subset of its nodes,
referred to as source nodes.
Each node in the graph corresponds to a communication unit.
The task is to disseminate a signal from the source nodes to all other nodes in a shortest possible time (broadcast time), while abiding by communication rules.
A node is said to be \emph{informed} at a given time if it is a source, or it already has received the signal from some other node.
Otherwise, the node is said to be \emph{uninformed}.
Consequently, the set of informed nodes is initially exactly the set of sources.
Reflecting the fact that communication can be established only between pairs of nodes that are located within a sufficiently close vicinity of each other,
the edge set of the graph consists of potential communication links along which the signal can be transmitted.
\change

Consider time represented by integers $1,2,\ldots$. \stopchange
Agreeing with the whispering model, every informed node can \change forward the signal to at most one uninformed neighbour node at a time. \stopchange
Therefore, the number of informed nodes \change is at most doubled at any time. \stopchange
This communication protocol appears in various practical applications, such as communication among computer processors or telephone networks.
In situations where the signals have to travel large distances, it is typically assumed that the signal is sent to one neighbour at \change a \stopchange time.
Inter-satellite communication networks thus constitute a prominent application area \citep{chu17}.
Particularly, the MBT problem arises when one or a few satellites need to broadcast data quickly by means of time-division multiplexing.
\change

\citet{lima22} mention several other industrial applications of MBT.
Noteworthy among these is a recent application in peer-to-peer network communication, in which significant improvements over a slow Bluetooth mesh were achieved.
According to \citet{lima22}, the problem under study also finds applications in wireless sensor networks \citep{shang10}, industry 4.0 \citep{hocaoglu19}, surveillance \citep{dekker02}, 
robotics \citep{bucantanschi07}, and direct memory access \citep{lazard92}.
\stopchange

The current literature on MBT offers some theoretical results, including complexity and approximability theorems.
Although inexact solution methods also have been proposed, few attempts seem to be made in order to compute the exact optimum,
or to find strong lower bounds on the minimum broadcast time.
The goal of the current text is to fill this gap, and we make the following contributions in that direction:
First, a compact integer \change linear programming (ILP) \stopchange model is developed.
\change
Unlike models applied in previous works \citep{desousa18,desousa18b,lima22}, the ILP model studied in the current text maximises the number of nodes that can be informed within
a given time $t$.
The optimal solution to the MBT problem is then identified as the minimum value of $t$ for which the objective function attains a value identical to the vertex cardinality of the graph.
With access to strong lower and upper bounds on the minimum broadcast time, such a model has to be run for only a few different values of $t$.
The current work demonstrates empirically that such an approach is, in a large proportion of available instances, considerably faster than solving the previously studied ILP models.

The benefit of the new ILP approach grows with increased strength of the bounds on the minimum broadcast time.
Our second contribution is a lower bounding technique, which proves its merit particularly in instances where all shortest paths from the source set to a non-source have moderate length.
Third, we devise an upper bounding algorithm, which in combination with strong lower bounds is able to close the optimality gap in a wide range of instances.
In summary, the current work contributes new methods for (1) exact estimates of, (2) lower bounds on, and (3) upper bounds on the minimum broadcast time.
\stopchange

The remainder of the paper is organised as follows:
Next, we review the current scientific literature on MBT and related problems,
and in Section \ref{sec:def}, a concise problem definition is provided.
The integer \change linear \stopchange program is formulated and discussed in Section \ref{sec:exact}.
Lower and upper bounding methods are derived in Sections \ref{sec:lb} and \ref{sec:ub}, respectively.
Computational experiments are reported in Section \ref{sec:exp}, before the work is concluded by Section \ref{sec:conc}.

\subsection{Literature overview} \label{sec:litrev}

Deciding whether an instance of MBT has a solution with broadcast time at most $t$ has been shown to be NP-complete \citep{garey79,slater81}. 
For bipartite planar graphs with maximum degree 3, NP-completeness persists even if $t=2$ or if there is only one source \citep{jansen95}.
When $t=2$, the problem also remains NP-complete for cubic planar graphs \citep{middendorf93}, grid graphs with maximum degree 3,
complete grid graphs, chordal graphs, and for split graphs \citep{jansen95}. 
The single-source variant of the decision version of MBT is NP-complete for grid graphs with maximum degree 4, and for chordal graphs \citep{jansen95}.
The problem is known to be polynomial in trees \citep{slater81}.
Whether the problem is NP-complete for split graphs with a single source was stated as an open question \change by \citet{jansen95}, \stopchange
and has to the best of our knowledge not been answered yet.

A number of inexact methods, for both general and special graph classes, have been proposed in the literature during the last three decades.
One of the first works of this category \citep{scheuermann84} 
introduces a dynamic programming algorithm that identifies all maximum matchings in an induced bipartite graph.
Additional contributions of \citet{scheuermann84} include heuristic approaches for near optimal broadcasting.
Among more recent works, \citet{hasson04} describe a \change metaheuristic \stopchange algorithm for MBT, and provide a comparison with other existing methods.
The communication model is considered in an existing satellite navigation system by \citet{chu17}, where a greedy inexact method is proposed together with a mathematical programming model.
Examples of additional efficient heuristics are contributed by e.g.\ \citet{harutyunyan14}, \citet{harutyunyan06}, \citet{lima22}, \citet{desousa18}, and \citet{wang10}.

Approximation algorithms for MBT are studied by \citet{kortsarz95}. 
The authors argue that methods presented by \citet{scheuermann84} provide no guarantee on the performance, and show that wheel-graphs are examples of unfavourable instances.
Another contribution from \citet{kortsarz95} is an $\mathcal{O}(\sqrt{n})$-additive approximation algorithm for broadcasting in general graphs with $n$ nodes.
The same work also provides approximation algorithms for several graph classes with small separators with approximation ratio proportional to the separator size times $\log n$.
An algorithm with $\mathcal{O}\left(\frac{\log n}{\log \log n}\right)$-approximation ratio is given by \citet{elkin03}.
(Throughout the current text, the symbol $\log$ refers to the logarithm with base 2.)
Most of the works cited above consider a single source.

A related problem extensively studied in the literature is the minimum broadcast graph problem \citep{grigni91,mcgarvey16}. 
A broadcast graph supports a broadcast from any node to all other nodes in optimal time $\lceil\log n\rceil$.
For a given integer $n$, a variant of the problem is to find a broadcast graph of $n$ nodes such that the number of edges in the graph is minimised.
In another variant, the maximum node degree rather than the edge cardinality is subject to minimisation.
\citet{mcgarvey16} study ILP models for $c$-broadcast graphs, which is a generalisation where signal transmission to at most $c$ neighbours \change at a time is allowed. \stopchange

Despite a certain resemblance with MBT, the minimum broadcast graph problem is clearly distinguished from \change the problem under study\stopchange,
and will consequently not be considered further in the current work.

\section{Network model and definitions} \label{sec:def}

The communication network is represented by a connected graph $G=(V,E)$ and a subset $S\subseteq V$ referred to as the set of \emph{sources}.
We denote the number of nodes and the number of sources by $n=|V|$ and $\sigma=|S|$, respectively. 
The digraph with nodes $V$ and arcs $(u,v)$ and $(v,u)$ for each $\{u,v\}\in E$ is denoted $\overrightarrow{G}=(V,\overrightarrow{E})$.
Finally, for a node $u\in V$, we define $N(u)=\left\{v\in V: \left\{u,v\right\}\in E\right\}$ as the set of neighbors of node $u$.

\begin{definition} \label{def:broadcasttime}
The \emph{minimum broadcast time} $\tau(G,S)$ of a node set $S\subseteq V$ in $G$ is defined as the smallest integer $t\geq 0$ for which there exist
a sequence $V_0\subseteq\dots\subseteq V_t$ of node sets and a function $\pi:V\setminus S\to V$, such that:
\begin{enumerate}
  \item $V_0=S$ and $V_t=V$, \label{def:boundary}
  \item for all $v\in V\setminus S$, $\{\pi(v),v\}\in E$, \label{def:edge}
  \item for all $k=1,\ldots,t$ and all $v\in V_k$, $\pi(v)\in V_{k-1}$, and \label{def:parent}
  \item for all $k=1,\ldots,t$ and all $u,v\in V_k\setminus V_{k-1}$, $\pi(u)=\pi(v)$ only if $u=v$. \label{def:unique}
\end{enumerate}
\end{definition}

Referring to Section \ref{intro}, the node set $V_k$ is the set of nodes that are informed \change at time \stopchange $k$.
Initially, only the sources are informed ($V_0=S$), whereas all nodes are informed after \change time $t$ \stopchange ($V_t=V$),
and the set of informed nodes is monotonously non-decreasing ($V_{k-1}\subseteq V_k$ for $k=1,\ldots,t$).
The parent function $\pi$ maps each node to the node from which it receives the signal.
Conditions \ref{def:edge}--\ref{def:parent} of Definition \ref{def:broadcasttime} thus reflect that the sender is a neighbour node in $G$,
and that it is informed at an earlier time than the recipient node.
Because each node can send to at most one neighbour node \change at a time\stopchange,
condition \ref{def:unique} states that $\pi$ maps the set of nodes becoming informed \change at time \stopchange $k$ to distinct parent nodes.
The preimage of $v$ under $\pi$, that is, the set of child nodes of $v$, is denoted $\pi^{-1}(v)$.

The optimisation problem in question is formulated as follows:
\begin{problem}[\textsc{Minimum Broadcast Time}]\label{prob:min}
Given $G=(V,E)$ and $S\subseteq V$, find $\tau(G,S)$.
\end{problem}

\begin{definition} \label{def:broadcastgraph}
For any $V_0,\ldots,V_t$ and $\pi$ satisfying the conditions of Definition \ref{def:broadcasttime}, possibly with the exception of $t$ being minimum,
the corresponding \emph{broadcast forest} is
the digraph $D=(V,A)$, where $A=\left\{\left(\pi(v),v\right): v\in V\right\}$.
If $t$ is minimum, $D$ is referred to as a \emph{minimum} broadcast forest.
Each connected component of $D$ is a \emph{communication tree}.
\end{definition}

\noindent
It is easily verified that the communication trees are indeed arborescences, rooted at distinct sources, with arcs pointing away from the source.
Let $T(s)=\left(V(s),A(s)\right)$ denote the communication tree in $D$ rooted at source $s\in S$,
and let $T_k(s)$ be the subtree of $T(s)$ induced by $V(s)\cap V_k$.
Analogously, let $D_k$ be the directed subgraph of $D$ induced by node set $V_k$.
For the sake of notational simplicity, the dependence on $(V_0,\ldots,V_t,\pi)$ is suppressed when referring to the directed graphs introduced here.

\noindent
The degree of node $v$ in graph $G$ is denoted $\deg_G(v)$.
For a given subset $U\subseteq V$ of nodes, we define $G[U]$ as the subgraph of $G$ induced by $U$.
We let $\deg^+_{\overrightarrow{G}}(v)$ and $\deg^-_{\overrightarrow{G}}(v)$ denote, respectively, the out-degree and the in-degree of node $v$ in $\overrightarrow{G}$,
and we let $\deg_{\overrightarrow{G}}(v)=\deg_{\overrightarrow{G}}^+(v)+\deg_{\overrightarrow{G}}^-(v)$.
When $p$ is a logical proposition, $\delta_p=1$ if $p$ is true, and $\delta_p=0$, otherwise.

\section{Exact methods} \label{sec:exact}

In this section, we formulate an ILP model for Problem \ref{prob:min}, and discuss possible solution strategies. 
\change
First, we give a multi-source version of the model suggested by \citet{desousa18} and pursued by \citet{lima22}, before we show how to formulate some of the constraints more strongly,
and how the decision version of the model can be exploited for faster convergence.

\subsection{Optimisation version: the broadcast time model of \citet{desousa18,desousa18b}}
\label{sec:optbasic}

Given integers $\underline{t}$ and $\bar{t}$ such that $\underline{t}\leq\tau(G,S)\leq\bar{t}$,
define the binary variables ($(u,v)\in \overrightarrow{E}$, $k=1,\ldots,\bar{t}$) 
$$ x_{uv}^k=
\begin{cases} 
1, \text{ if } v\in V_k\setminus V_{k-1} \text{ and } \pi(v)=u,\\ 
0, \text{ otherwise.}
\end{cases}
$$
\noindent
The \stopchange variable $x_{uv}^k$ thus represents the decision whether or not the signal is to be transmitted from node $u$ to node $v$ at time $k$.
\change
Let $z$ be an integer variable representing the broadcast time.

Possibly weak bounds $\underline{t}$ and $\bar{t}$ on the broadcast time $\tau(G,S)$ are easily available. \stopchange
Because $G$ is connected, the cut between any set $V_i$ of informed nodes and its complement is non-empty,
and therefore at least one more node can be informed \change at any time \stopchange.
It follows that $\tau(G,S)\leq n-\sigma$.
The bound is tight in the worst case instance where $S=\{v_1\}$, and $G$ is a path with $v_1$ as one of its end nodes.
\change Further, $\tau(G,S)\geq\delta_{V\setminus S\neq\emptyset}$ is a trivial lower bound.
Problem \ref{prob:min} is formulated as follows \citep{desousa18,desousa18b}: 
\begin{subequations}\label{mod:basic}
\begin{align}
\label{mod:basic:sousa1} \min z, \\ 
\text{s. t.~~~} \label{mod:basic:sousa2} \sum\limits_{v\in N(u)}x_{uv}^1  \leq \delta_{u\in S}, && u\in V,\\
\label{mod:basic:sousa4} \sum\limits_{v\in N(u)}x_{uv}^k \leq 1, &&  u\in V, k=2,\dots,\bar{t},\\
\label{mod:basic:sousa3} \sum\limits_{v\in N(u)}\sum\limits_{k=1}^{\bar{t}}x_{vu}^k = 1, && u\in V\setminus S,\\
\label{mod:basic:sousa5} x_{uv}^k \leq \sum\limits_{\ell=1}^{k-1}\sum\limits_{w\in N(u)\setminus\{v\}}x_{wu}^{\ell}, && u\in V\setminus S, v\in N(u), k=2,\dots,\bar{t},\\
\label{mod:basic:sousa6} z \geq \sum\limits_{k=1}^{\bar{t}}kx_{uv}^k, && u\in V, v\in N(u),\\
\label{mod:basic:sousa78}x_{uv}^k\in\{0,1\}, \underline{t}\leq z\leq\bar{t}, && u\in V, v\in N(u), k=1,\ldots,\bar{t}.&&
\end{align}~
\end{subequations}
By \eqref{mod:basic:sousa2}, every source (every non-source) node $u$ sends the signal to at most one neighbor node $v$ (does not send at all) at time $1$.
Analogously, constraints \eqref{mod:basic:sousa4} state that no node can send to more than one neighbor at a time later than 1.
Constraints \eqref{mod:basic:sousa3} ensure that all nodes eventually get informed.
The requirement that a non-source node $u$ informs a neighbour $v$ at time $k$ only if $u$ is informed by some adjacent node $w$
at an earlier time is modeled by \eqref{mod:basic:sousa5}. 
Lastly, constraints \eqref{mod:basic:sousa6} enforce the broadcast time variable $z$ to take a value no less than $k$ if transmissions take place at time $k$.

\subsection{Decision version: maximising the number of informed nodes}
\label{sec:decbasic}
\stopchange
The nature of MBT suggests another modelling approach \change based on a subset of the binary variables in \stopchange model \eqref{mod:basic}. 
For an integer $t\in[\underline{t},\bar{t}]$, let $\nu(t)$ denote the maximum number of non-source nodes that can be informed within time $t$, which means that
$\tau(G,S)=\min\left\{t: \nu(t)=n-\sigma\right\}$.
Hence, $\tau(G,S)$ is found by evaluating $\nu(t)$ for $t=\underline{t},\ldots,\bar{t}-1$, interrupted by the first occurrence of $\nu(t)=n-\sigma$.
In the worst case, it is observed that $\nu(\bar{t}-1)<n-\sigma$, which leads to the conclusion $\tau(G,S)=\bar{t}$.
The tightness of the upper and lower bound largely affects \change the \stopchange computational efficiency of this procedure.
Clearly, the lower bound $\underline{t}$ allows the \change omission of \stopchange the iterations for $t<\underline{t}$. 
Also, if \change $\nu(t)<n-\sigma$ is observed for $t=\bar{t}-1$, it is concluded \stopchange that $\tau(G,S)=\bar{t}$, and so the iteration for $t=\bar{t}$ does not have to be performed.
\change

Let $sp_u$ denote the number of edges on the shortest path in $G$ from $S$ to node $u\in V$.
Obviously, $u$ is informed no earlier than time $sp_u$, and at earliest it informs a neighbor node $v$ at time $sp_u+1$.
Let the binary variable $x_{uv}^k$ be defined for all $k=sp_u+1,\ldots,t$, and let $x_{uv}^k=0$ for $k\leq sp_u$.
Observe that every MBT instance has an optimal solution where no node $u$ is idle at some time $k\in(sp_u,t)$, while informing some neighbor node at time $k+1$.
To reduce the redundancy in the model, this observation is exploited in the decision version.
Further, the number of constraints is reduced from $\Theta\left(\bar{t}|\overrightarrow{E}|\right)$ (see constraints \eqref{mod:basic:sousa5}) to
$\Theta\left(t|V|\right)$ in the following formulation of the decision problem:

\begin{subequations}\label{mod:basic:dec}
\begin{align}
\label{mod:basic:dec:obj} \nu(t) = \max \sum\limits_{v \in V\setminus S}\sum\limits_{u \in N(v)} \sum\limits_{k=sp_u+1}^{t}x_{uv}^k, \\ 
\text{s. t.~~~} \label{mod:basic:dec:atMost1in} \sum\limits_{v\in N(u)}\sum\limits_{k=sp_v+1}^{t}x_{vu}^k  \leq 1, && u\in V \setminus S,\\
\label{mod:basic:dec:c} \sum\limits_{v\in N(u)}x_{uv}^k \leq\sum\limits_{v\in N(u):sp_v<k-1}x_{vu}^{k-1}+\sum\limits_{v\in N(u)}x_{uv}^{k-1}, && u\in V\setminus S, k=sp_u+1,\dots,t,\\
\label{mod:basic:dec:tcrel1} \sum\limits_{v\in N(u)}x_{uv}^1  \leq 1, &&  u\in S,\\
\label{mod:basic:dec:tcrel2} \sum\limits_{v\in N(u)}x_{uv}^k  \leq \sum\limits_{v\in N(u)}x_{uv}^{k-1}, &&  u\in S, k=2,\ldots,t,\\
\label{mod:basic:dec:dim}x_{uv}^k \in \{0,1\}, && (u,v)\in\overrightarrow{E}, k=sp_u+1,\dots,t.
\end{align}~
\end{subequations}
\stopchange

In the transition from the optimisation model \eqref{mod:basic}, constraints \eqref{mod:basic:sousa3} are replaced by \eqref{mod:basic:dec:atMost1in}.
The constraints are inequalities in the decision version, because \change some nodes may be left uninformed at time $t$.
Constraints \eqref{mod:basic:dec:c} state that node $u$ informs a neighbor at time $k>sp_u$ only if it either did so also at time $k-1$ or received the signal at that time.
It follows from $x_{uv}^{sp_u}=0$ and \eqref{mod:basic:dec:atMost1in} that the right hand side of \eqref{mod:basic:dec:c} is at most 1 for $k=sp_u+1$.
A simple induction argument shows that $\sum\limits_{v\in N(u)}x_{uv}^k \leq 1$ for all $k=sp_u+1,\ldots,t$, and hence \eqref{mod:basic:sousa4} is satisfied.
Likewise, summating \eqref{mod:basic:dec:c} over time yields $\sum\limits_{v\in N(u)}x_{uv}^k\leq\sum\limits_{v\in N(u)}\sum\limits_{\ell=sp_v+1}^{k-1}x_{vu}^{\ell}$,
ensuring that $u$ is informed before informing others.
Because the right hand side of \eqref{mod:basic:dec:c} is no larger than its counterpart in \eqref{mod:basic:sousa4}, \eqref{mod:basic:dec:c} is at least as strong as \eqref{mod:basic:sousa4}.
The constraints \eqref{mod:basic:dec:tcrel1}--\eqref{mod:basic:dec:tcrel2} stating that each source informs at most one neighbor at a time are formulated analogously.
\stopchange
In summary, $\tau(G,S)$ is computed by the following procedure:

\begin{algorithm}
\KwData{$G$, $S$, $\underline{t}$, $\bar{t}$}
\For{$t=\underline{t},\ldots,\bar{t}-1$} {
	Compute $\nu(t)$ by solving \eqref{mod:basic:dec} \\
	\textbf{if} ~$\nu(t)=n-\sigma$ ~\textbf{return} $t$ \\
}
\textbf{return} $\bar{t}$ \\
\caption{Exact solution to \textsc{Minimum Broadcast Time}}
\label{alg:ip}
\end{algorithm}

\begin{remark} \label{rem:ip}
If Alg.\ \ref{alg:ip} is interrupted due to an imposed time limit when processing an integer $t\in[\underline{t},\bar{t})$, 
	the broadcast time $\tau(G,S)$ is not known, but a (possibly tighter) lower bound $t$ on $\tau(G,S)$ is identified.
\end{remark}

\begin{remark} \label{rem:binary}
Algorithm \ref{alg:ip} follows the principle of \emph{sequential search}.
While a worst-case analysis suggests that \emph{binary search} concludes in fewer iterations, sequential search is favoured by smaller ILP instances to be solved.
The number of variables and constraints in \eqref{mod:basic:dec} increases linearly with $t$, and the time needed to compute $\nu(t)$ is thus expected to grow exponentially with $t$.
\end{remark}

\section{Lower bounds} \label{sec:lb}
Strong lower bounds on the minimum objective function value are, \change in general\stopchange, of vital importance to combinatorial optimisation algorithms.
\change Algorithm \ref{alg:ip} benefits directly from the bound $\underline{t}\leq\tau(G,S)$ by omitting calculations of $\nu(t)$ for $t<\underline{t}$. \stopchange
In this section, we study three types of lower bounds on the broadcast time $\tau(G,S)$.

\subsection{Analytical lower bounds} \label{sec:lbanalyt}
Any solution $\left(V_0,\ldots,V_t,\pi\right)$ satisfying conditions \ref{def:boundary}--\ref{def:unique} of Definition \ref{def:broadcasttime},
also satisfies $\left|V_{k+1}\right|\leq 2\left|V_k\right|$ for all $k\geq 0$.
\change Because the signal is passed along some path from $S$ to node $v\in V\setminus S$, and the length of the path is at least $sp_v$,
node $v$ becomes informed at no earlier time than $sp_v$
\citep[Theorem 1]{lima22}.
This yields the following lower bound:
\begin{observation}
\begin{equation}
\max\left\{\left\lceil\log\frac{n}{\sigma}\right\rceil, \max_{v\in V\setminus S}sp_v\right\}\leq \tau(G,S).
\label{eq:loglb}
\end{equation}
\label{obs:loglb}
\end{observation}
\stopchange

Consider the $m$-step Fibonacci numbers $\left\{f^{m}_k\right\}_{k=1,2,\ldots}$ \citep{noe05}, a generalisation of the well-known (2-step) Fibonacci numbers, defined by
$f^{m}_k=0$ for $k\leq 0$, $f^{m}_1=1$, and 
other terms according to the linear recurrence relation 
\begin{align*}
f^{m}_k &=\sum\limits_{j=1}^m f^{m}_{k-j}, &\text{ for } k\geq 2.
\end{align*}
\change

\begin{observation} \label{obs:fib}
$f^{m}_k=2^{k-2}$ for $k=2,\ldots,m+1$.
\end{observation}
\stopchange

The generalised Fibonacci numbers are instrumental in the derivation of a lower bound on $\tau(G,S)$,
depending on the maximum node degree $d=\max\left\{\deg_G(v): v\in V\right\}$ in $G$.
The idea behind the bound is that the broadcast time can be no shorter than what is achieved if
the following ideal, but not necessarily feasible, criteria are met:
Every source transmits the signal to a neighbour node \change at every time \stopchange $1,\ldots,d$,
and every node $u\in V\setminus S$
transmits the signal to a neighbour node in each of the first $d-1$ periods following the \change time \stopchange when $u$ gets informed.
An exception possibly occurs in the last period, as there may be fewer nodes left to be informed than there are nodes available to inform them.

\begin{proposition}
\begin{equation*}
\label{lem:lbreg1}
	\tau(G,S)\geq\min\left\{t:2\sigma\sum\limits_{j=1}^tf^{d-1}_j\geq n\right\}.
\end{equation*}
\label{prop:lbfib}
\end{proposition}
\begin{proof}

Consider a solution $\left(V_0,\ldots,V_t,\pi\right)$ with associated broadcast graph $D$, such that $V_0\subsetneq V_1\subsetneq\cdots\subsetneq V_{t-1}\subsetneq V_t$, 
\begin{itemize}
  \item conditions \ref{def:boundary} and \ref{def:parent}--\ref{def:unique} of Definition \ref{def:broadcasttime} are satisfied,
  \item for each source $u\in S$ and each $j=1,\ldots,\min\{d,t-1\}$, there exists a node $v\in V_j\setminus V_{j-1}$ such that $\pi(v)=u$, and
  \item for each $k\in\{1,\ldots,t-2\}$, each node $u\in V_k\setminus V_{k-1}$, and each $j=k+1,\ldots,\min\{k+d-1,t-1\}$,
        there exists a node $v\in V_j\setminus V_{j-1}$ such that $\pi(v)=u$.
\end{itemize}
\noindent
That is, all sources send the signal to some uninformed node (not necessarily a neighbour node) \change at all times \stopchange up to $\min\{d,t-1\}$.
All nodes that received the signal \change at time \stopchange $k$, forward it to some uninformed node \change at all times \stopchange up to $\min\{d-1,t-1\}$,
and all nodes are informed \change at time \stopchange $t$.
Because condition \ref{def:edge} of Definition \ref{def:broadcasttime}, stating that the flow of information follows $E$, is not imposed,
such a solution $\left(V_0,\ldots,V_t,\pi\right)$ exists for an appropriate choice of $t$.
Since the solution implies that every node is actively receiving or sending for up to $d$ consecutive periods, until the signal is broadcast \change at time \stopchange $t$,
it follows that $\tau(G,S)\geq t$.
It remains to prove that the chosen $t$ is the smallest value satisfying $2\sigma\sum_{k=1}^tf_k^{d-1}\geq n$, i.e.,
that $2\sigma\sum_{k=1}^{t-1}f_k^{d-1}<n\leq 2\sigma\sum_{k=1}^tf_k^{d-1}$.

For $k=1,\ldots,t$, let $L_k=\left\{v\in V_k:\deg_{D_k}(v)=1\right\}$ denote the set of nodes with exactly one out- or in-neighbour in $D_k$,
and let $L_k=\emptyset$ for $k\leq 0$.
That is, for $k>1$, $L_k$ is the set of nodes that receive the signal \change at time \stopchange $k$,
whereas $L_1$ consists of all nodes informed \change at time \stopchange 1, including the sources $S$.
Hence, $L_1,\ldots,L_{t-1}$ are disjoint sets (but $L_t$ may intersect $L_{t-1}$), and $V_k=L_1\cup\cdots\cup L_k$ for all $k=1,\ldots,t$.

Consider a \change time \stopchange $k\in\{2,\ldots,t-1\}$.
The assumptions on $\left(V_0,\ldots,V_t,\pi\right)$ imply that $\pi$ is a bijection from $L_k$ to $L_{k-1}\cup\cdots\cup L_{k-d+1}$.
Thus, $\left|L_k\right|=\sum_{j=1}^{d-1}\left|L_{k-j}\right|$.
Since also $\left|L_1\right|=2\sigma=2\sigma f_1^{d-1}$ and $\left|L_j\right|=f_j^{d-1}=0$ for $j\leq 0$,
we get $\left|L_k\right|=2\sigma f_k^{d-1}$.
Further, $\left|L_t\right|\leq\sum_{j=1}^{d-1}\left|L_{t-j}\right|=2\sigma f_t^{d-1}$.
It follows that $2\sigma\sum_{k=1}^{t-1}f_k^{d-1}=\sum_{k=1}^{t-1}\left|L_k\right|=\left|V_{t-1}\right|<n=\left|V_t\right|\leq\sum_{k=1}^t\left|L_k\right|\leq 2\sigma\sum_{k=1}^tf_k^{d-1}$,
which completes the proof.
\end{proof}

\subsection{Combinatorial relaxations} \label{sec:lbcombrel}

Lower bounds on the broadcast time $\tau(G,S)$ are obtained by replacing one or more of the conditions imposed in Definition \ref{def:broadcasttime}
by more lenient conditions.
Because condition \ref{def:edge} states that source $s\in S$ is the parent of $v$ only if $\{s,v\}\in E$,
the condition implies that $s$ has no more than $\deg_G(s)$ child nodes.
Analogously, for any $u\in V\setminus S$, the condition implies that $u$ has at most $\deg_G(u)-1$ child nodes. 
As the implications do not apply in the reverse direction, a relaxation is obtained if
condition \ref{def:edge} is replaced by
\begin{enumerate}
\setcounter{enumi}{4}
  \item for all $v\in V$, $\left|\pi^{-1}(v)\right|\leq\deg_G(v)-\delta_{v\in V\setminus S}$. \label{def:degree}
\end{enumerate}

\noindent
A lower bound on $\tau(G,S)$ is then given by the solution to:
\begin{problem}[\textsc{Node Degree Relaxation}]\label{prob:degree}
Find the smallest integer $t\geq 0$ for which there exist
a sequence $V_0\subseteq\dots\subseteq V_t$ of node sets and a function $\pi:V\setminus S\to V$,
satisfying conditions \ref{def:boundary} and \ref{def:parent}--\ref{def:degree}.
\end{problem}

Observe that the bound given in Proposition \ref{prop:lbfib} is obtained by exploiting the lower-bounding capabilities of the \textsc{Node Degree Relaxation}.
By considering the degree of all nodes $v\in V$, rather than just the maximum degree, stronger bounds may be achieved in instances where $G$ is not regular
($\min_{v\in V}\deg_G(v)<\max_{v\in V}\deg_G(v)$).

Denote the source nodes $S=\left\{v_1,\dots,v_{\sigma}\right\}$ and the non-source nodes $V\setminus S=\left\{v_{\sigma+1},\ldots,v_n\right\}$,
where $\deg_G(v_{\sigma+1})\geq\deg_G(v_{\sigma+2})\geq\dots\geq\deg_G(v_n)$,
and let $d_i=\deg_G(v_i)$ ($i=1,\ldots,n$).
Thus, $\left\{d_1,\ldots,d_n\right\}$ resembles the conventional definition of a non-increasing degree sequence of $G$,
with the difference that only the subsequence consisting of the final $n-\sigma$ degrees is required to be non-increasing.

For a given $t\in\mathbb{Z}_+$, consider the problem of finding $\left(V_0,\ldots,V_t,\pi\right)$ such that $V_0=S$,
conditions \ref{def:parent}--\ref{def:degree} are satisfied, and $\left|V_t\right|$ is maximised.
The smallest value of $t$ for which the maximum equals $n$ is obviously the solution to Problem~\ref{prob:degree}.

The algorithm for Problem~\ref{prob:degree}, to follow later in the section, utilises that the maximum value of $\left|V_t\right|$
is achieved by transmitting the signal to nodes in non-increasing order of their degrees.
Observe that, contrary to the case of Problem~\ref{prob:min}, transmissions to non-neighbours are allowed in the relaxed problem.
Any instance of Problem~\ref{prob:degree} thus has an optimal solution where, for $k=1,\ldots,t-1$,
$u\in V_k\setminus V_{k-1}$ and $v\in V_{k+1}\setminus V_k$ implies $\deg_G(u)\geq\deg_G(v)$.

A rigorous proof of this follows next.

\begin{lemma}
\label{lemma:degorder}
The maximum value of $\left|V_t\right|$ over all $\left(V_0,\ldots,V_t,\pi\right)$ satisfying $V_0=S$ and
conditions \ref{def:parent}--\ref{def:degree}, is attained by some
$\left(V_0,\ldots,V_t,\pi\right)$ where
$\min\left\{i: v_i\in V_{k}\setminus V_{k-1}\right\}>\max\left\{i: v_i\in V_{k-1}\right\}$ ($k=1,\ldots,t$).
\end{lemma}
\begin{proof}

Consider an arbitrary optimal solution $\left(V_0,\ldots,V_t,\pi\right)$,
and assume that $v_i\in V_p\setminus V_{p-1}$, $v_j\in V_{q}\setminus V_{q-1}$, $i<j$, and $1\leq q<p\leq t$.
We prove that the solution 
obtained by swapping nodes $v_i$ and $v_j$ is also optimal.
Let $\bar{V}_k=V_k$ for $k=0,\ldots,q-1, p, p+1,\ldots,t$, and $\bar{V}_k=\left(V_k\setminus\{v_j\}\right)\cup\{v_i\}$ for $k=q,\ldots,p-1$.
Because $\left|\bar{V}_t\right|=\left|V_t\right|$, we only need to show that $\left(\bar{V}_0,\ldots,\bar{V}_t,\bar{\pi}\right)$ is feasible for some $\bar{\pi}$.
In the following, we demonstrate that a valid parent function $\bar{\pi}$ can be obtained by swapping $\pi(v_i)$ and $\pi(v_j)$,
along with a simple adjustment ensuring that $\left|\bar{\pi}^{-1}(v_j)\right|\leq\left|\pi^{-1}(v_j)\right|$.

Define $m=\max\left\{0,\left|\pi^{-1}(v_i)\right|-\left|\pi^{-1}(v_j)\right|\right\}$.
Consider the case where $m>0$.
Because $v_i$ has at most one child in each $V_k\setminus V_{k-1}$ ($k=p+1,\ldots,t$),
there exist integers $p_1>\cdots>p_m>p$, and nodes $u_r\in V_{p_r}\setminus V_{p_r-1}$ ($r=1,\ldots,m$) such that $\pi(u_r)=v_i$, whereas $v_j$ has no child in
$\bigcup_{r=1}^m\left(V_{p_r}\setminus V_{p_r-1}\right)$.
Let $U=\{u_1,\ldots,u_m\}$, and let $U=\emptyset$ if $m=0$.

Let $\bar{\pi}(v)=v_i$ for all $v\in U$, and $\bar{\pi}(v)=v_j$ for all $v\in \pi^{-1}(v_i)\setminus U$.
Also, let $\bar{\pi}(v)=v_i$ for all $v\in \pi^{-1}(v_j)\setminus\{v_i\}$.
If $\pi(v_i)=v_j$, let $\bar{\pi}(v_j)=v_i$, otherwise let $\bar{\pi}(v_j)=\pi(v_i)$.
Let $\bar{\pi}(v_i)=\pi(v_j)$.
For all other non-source nodes, that is, all $v\in V\setminus S$ for which $v_i\neq\pi(v)\neq v_j$, let $\bar{\pi}(v)=\pi(v)$.

If $m>0$, $\left|\bar{\pi}^{-1}(v_i)\right|=\left|\pi^{-1}(v_i)\right|\leq\deg_G(v_i)-1$
and $\left|\bar{\pi}^{-1}(v_j)\right|=\left|\pi^{-1}(v_j)\right|\leq\deg_G(v_j)-1$.
Otherwise, $\left|\bar{\pi}^{-1}(v_i)\right|=\left|\pi^{-1}(v_j)\right|\leq\deg_G(v_j)-1\leq\deg_G(v_i)-1$,
and $\left|\bar{\pi}^{-1}(v_j)\right|=\left|\pi^{-1}(v_i)\right|\leq\left|\pi^{-1}(v_j)\right|\leq\deg_G(v_j)-1$.
For $v_i\neq v\neq v_j$, $\left|\bar{\pi}^{-1}(v)\right|=\left|\pi^{-1}(v)\right|$,
and thus $\left(\bar{V}_0,\ldots,\bar{V}_t,\bar{\pi}\right)$ satisfies condition \ref{def:degree}.
It is straightforward to show that $\left(\bar{V}_0,\ldots,\bar{V}_t,\bar{\pi}\right)$ also satisfies conditions \ref{def:parent}--\ref{def:unique}.
\end{proof}

Algorithm \ref{alg:dreg} takes as input the number $\sigma$ of sources and the number $n$ of nodes, along with the node degrees $d_1,\ldots,d_n$,
where $d_{\sigma+1}\geq\cdots\geq d_n$.
It operates with \change counters $\nu_t$ \stopchange of informed nodes \change at time $t$\stopchange, initiated to \change $\nu_0=\sigma$\stopchange.
Thus, nodes $v_1,\ldots,v_{\nu_t}$ are informed \change at time $t$\stopchange, whereas $\change v_{\nu_t+1}\stopchange,\ldots,v_n$ are not.
A counter denoted $a_i$ ($i=1,\ldots,n$) keeps track of the number of nodes informed by node $v_i$.
The sets $F_{\change t\stopchange}$ consists of indices $i$ of informed nodes that \change at time $t$ \stopchange have not sent the signal to $d_i-1$ nodes ($d_i$ nodes if $i\leq\sigma$). 
In each iteration of the outer loop of the algorithm, all nodes $v_i$ for which $i\in F_{\change t\stopchange}$ inform some currently uninformed node,
and all counters are updated accordingly.
The process stops when all $n$ nodes are informed, and the number of performed iterations is returned.

\begin{algorithm}
\KwData{$\sigma,n,d_1,\ldots,d_n\in\mathbb{Z}_+$}
$a_1\leftarrow\cdots\leftarrow a_n\leftarrow 0$, $\nu_{\change 0\stopchange}\leftarrow\sigma$ \\
\For{$t=1,2,\dots$} {
	$F_{\change t\stopchange}\leftarrow\left\{i=1,\ldots,\nu_{\change t-1\stopchange}: a_i<d_i-\delta_{i>\sigma}\right\}$\\
	\change $\nu_{\change t\stopchange}\leftarrow\nu_{\change t-1\stopchange}+|F_t|$ \stopchange \\
	\textbf{if} ~$\nu_{\change t\stopchange}\geq n$ ~\textbf{return} $t$ \\
	\textbf{for} ~$i\in\change F_t\stopchange$ ~\textbf{do} ~$a_i\leftarrow a_i+1$ \\
}
\caption{Lower bound exploiting the degree distribution}
\label{alg:dreg}
\end{algorithm}

\begin{proposition}
Algorithm \ref{alg:dreg} returns a lower bound on $\tau(G,S)$.
\label{cor:deg}
\end{proposition}
\begin{proof}
		Follows from Lemma \ref{lemma:degorder} and the subsequent discussion. 
\end{proof}
\change

It is next proved that the lower bound produced by Alg.\ \ref{alg:dreg}, henceforth denoted $\eta$, is no weaker than the Fibonacci bound (Proposition \ref{prop:lbfib}) and the logarithmic bound.

\begin{proposition}\label{prop:strength}
$\eta\geq\max\left\{\min\left\{t:2\sigma\sum\limits_{j=1}^tf^{d-1}_j\geq n\right\},\left\lceil\log\frac{n}{\sigma}\right\rceil\right\}$.
\end{proposition}
\begin{proof}
That $\eta\geq\left\lceil\log\frac{n}{\sigma}\right\rceil$ follows immediately from $\nu_t=\nu_{t-1}+|F_t|\leq 2\nu_{t-1}$.
Because $\eta\geq\eta'$, where $\eta'$ is the output from Alg.\ \ref{alg:dreg} when the input data is $(\sigma,n,d,\ldots,d)$ (recall that $d=\max\left\{d_i: i=1,\ldots,n\right\}$),
it suffices to prove that $\eta'=\min\left\{t:2\sigma\sum\limits_{j=1}^tf^{d-1}_j\geq n\right\}$.
To that end, assume $d_1=\cdots=d_n=d$.
Then, $\left|F_t\right|=2^{t-1}\sigma$ for $t=1,\ldots,d$, and by Observation \ref{obs:fib}, $\left|F_t\right|=2\sigma f_t^{d-1}$ for $t=2,\ldots,d$.
For $t>d$, $\left|F_t\right|=\sum_{j=1}^{d-1}\left|F_{t-j}\right|$, which shows that $\left|F_t\right|$ is given by the recurrence formula of the $(d-1)$-step Fibonacci sequence.
Hence, $\left|F_t\right|=2\sigma f_t^{d-1}$ for all $t\geq 2$.
Since also $\nu_0+\left|F_1\right|=2\sigma=2\sigma f_1^{d-1}$, we get $\nu_t=\nu_0+\sum_{j=1}^t\left|F_j\right|=2\sigma\sum_{j=1}^tf_j^{d-1}$.
It follows that $\eta$ is the smallest value of $t$ for which $2\sigma\sum_{j=1}^tf_j^{d-1}\geq n$, which completes the proof.
\end{proof}
\stopchange

\section{Upper bounds} \label{sec:ub}

Access to an upper bound $\bar{t}\geq\tau(G,S)$ affects the number of variables in the models studied in Sections \ref{sec:optbasic}--\ref{sec:decbasic}.
Algorithms that output feasible, or even near-optimal solutions, are instrumental in the computation of upper bounds.
Further, such methods are required in sufficiently large instances, where exact approaches fail to terminate within a practical time.

\subsection{Existing heuristic methods} \label{sec:heur}
Building on earlier works \citep{harutyunyan06, harutyunyan10},
\citet{harutyunyan14} study a heuristic (considering $\sigma=1$) departing from a shortest-path tree of $G$.
A sequence of local improvements is performed in the bottom-up direction in the tree, starting by the leafs and terminating at the root node.
Rearrangements of the parent assignments are made in order to reduce the broadcast time needed in subtrees.
The heuristic has running time $\mathcal{O}\left(|E|\log{n}\right)$.
\change

Alternative heuristic methods have been studied by \citet{desousa18,desousa18b}.
\citet{hasson04} further suggest a metaheuristic belonging to the ant colony paradigm.
More recently, \citet{lima22} report comprehensive numerical experiments with a random-key genetic algorithm,
and provide empirical evidence of competitive computational performance of their method.
\stopchange

\subsection{A construction method} \label{sec:cons}

Consider an integer $t'\geq 0$, node sets $S=V_0\subseteq V_1\subseteq\cdots\subseteq V_{t'}\neq V$ and a function $\pi: V\setminus S\to V$,
where $\left\{\pi(v),v\right\}\in E$ for all $v\in V_{t'}\setminus S$,
and conditions \ref{def:parent}--\ref{def:unique} of Definition \ref{def:broadcasttime} are satisfied for $t=t'$.
That is, $\left(V_0,\ldots,V_{t'},\pi\right)$ defines a broadcast forest corresponding to the instance $\left(G\left[V_{t'}\right],S\right)$,
but the forest does not cover $V$.
In particular, if $t'=0$, the broadcast forest is a null graph on $S$, while it is a matching from $S$ to $V_1\setminus S$ if $t'=1$.

This section addresses the problem of extending the partial solution $\left(V_0,\ldots,V_{t'},\pi\right)$ by another node set $V_{t'+1}$,
such that the conditions above also are met for $t=t'+1$.
With $t'=0$ as departure point, a sequence of extensions results in a broadcast forest corresponding to instance $(G,S)$.
Each extension identifies a matching from $V_{t'}$ to $V\setminus V_{t'}$, and all matched nodes in the latter set are included in $V_{t'+1}$.
A key issue is how to determine the matching.

Since the goal is to minimise the time (number of extensions) needed to cover $V$, a \emph{maximum cardinality} matching between
$V_{t'}$ and $V\setminus V_{t'}$ is a natural choice.
Lack of consideration of the matched nodes' capabilities to inform other nodes is however an unfavourable property.
Each iteration of Alg.\ \ref{alg:match} rather sees $\change\kappa\stopchange\geq 1$ time periods ahead, and maximises the total number of nodes in $V\setminus V_{t'}$
that can be informed \change at time \stopchange $t'+1,\ldots,t'+\change\kappa\stopchange$.
Commitment is made for only one period, and the matched nodes are those that are informed \change at time \stopchange $t'+1$ from some node in $V_{t'}$.
The maximisation problem in question is exactly the one addressed by model \eqref{mod:basic:dec},
where $V_{t'}$ is considered as sources, $\change\kappa\stopchange$ the upper bound on the broadcast time, and the graph is $G$ with all edges within $V_{t'}$ removed.
Choosing $\change\kappa\stopchange=1$ corresponds to the maximum \change cardinality \stopchange matching option.

\begin{remark} \label{rem:exact}
Algorithm \ref{alg:match} is developed into an \emph{exact method} by choosing $\change\kappa\stopchange$ at least as large as any available upper bound on $\tau(G,S)$.
If the algorithm returns a value $\ell\leq\change\kappa\stopchange$, it follows that $\tau(G,S)=\ell$.
\end{remark}

\begin{algorithm}[]
	\KwData{$G=(V,E), S\subseteq V, \change\kappa\stopchange\in \{1,\dots,n-\sigma\}$}
	$S'\leftarrow S$,
	$F\leftarrow\emptyset$\\
\While{$S'\neq V$} {
	$x\leftarrow$ an optimal solution to the instance $(G,S',\change\kappa\stopchange)$ of \change the ILP \eqref{mod:basic:dec} \stopchange \\
	\For{$\{u,v\}\in E$ such that $u\in S'$, $v\in V\setminus S'$, and $x_{uv}^1=1$} {
		$S'\leftarrow S'\cup\{v\}$,
		$F\leftarrow F\cup\left\{\change (\stopchange u,v\change )\stopchange \right\}$\\
	}
}
	\textbf{return} $\tau((V,F),S)$ \\
\caption{Computing an upper bound on $\tau(G,S)$ through sequences of matchings}
\label{alg:match}
\end{algorithm}

Algorithm \ref{alg:match} generates a broadcast forest $(V,F)$ consisting of $|S|$ trees rooted at distinct sources.
The broadcast time $\tau\left((V,F),S\right)$ of the forest is thus an upper bound on $\tau(G,S)$.
\change

In many instances, \eqref{mod:basic:dec} has multiple optimal solutions.
Which of these is assigned to $x$ in line 3 may affect the bound eventually returned by Alg.~\ref{alg:match}.
Favourable tie breaking can be approached heuristically, e.g., by
\begin{itemize}
  \item discouraging $x_{uv}^1=1$ if $\gamma_u=\text{dist}_u+\text{child}_u$ is large, where $\text{dist}_u$ is the distance from $S$ to $u$ in the directed forest $(S',F)$,
        and $\text{child}_u$ is the number of child nodes of $u$ in the forest,
  \item and encouraging $x_{uv}^1=1$ if $\zeta_v=\left|N(v)\setminus S'\right|$ is large.
\end{itemize}
\noindent
Motivation for the former rule is found in the observation that the value $\tau\left((V,F),S\right)$ returned from Alg.\ \ref{alg:match} is no smaller than
$\max_{u\in V}\left(\text{dist}_u+\text{child}_u\right)$.
Moreover, including in $S'$ a node $v$ with a large neighbourhood in $V\setminus S'$ is preferable to including one for which $\left|N(v)\setminus S'\right|$ is small,
as such a choice implies a larger cut set between $S'$ and $V\setminus S'$.
The larger the cut set, the more edges there are for the algorithm to choose from in subsequent iterations.
Letting $c_{uv}^k=1-\gamma_u/|V|+\zeta_v/|E|$ if $k=1$, and $c_{uv}^k=1$ otherwise, and multiplying $x_{uv}^k$ by $c_{uv}^k$ in the objective function of model \eqref{mod:basic:dec}
yields the desired tie breaking.
It is readily verified that by the modest weight on $\gamma_u$ and $\zeta_v$, optimality is preserved for at least one optimal solution to \eqref{mod:basic:dec}.
\stopchange

\begin{remark} \label{rem:time}
If $\change\kappa\stopchange=1$, then the running time of Alg.\ \ref{alg:match} is $\mathcal{O}\left(n^{\frac{3}{2}}|E|\right)$, because the number of iterations is no more than $n$,
and the maximum cardinality matching is found in $\mathcal{O}\left(\sqrt{n}|E|\right)$ time \citep{hopcroft73}.
By applying the algorithm by \change \citet{proskurowski81} \stopchange for computing the broadcast time of a tree, $\tau\left((V,F),S\right)$ is computed in linear time.
For fixed $\change\kappa\stopchange\geq 2$, the problem solved in each iteration is NP-hard \citep{jansen95}, and the running time of Alg.\ \ref{alg:match} is exponential.
\end{remark}

\begin{remark} \label{rem:mvm}
\change
If $\change\kappa\stopchange=1$, the tie breaking rule in terms of a modified objective function indicated above implies that
maximum cardinality matching is replaced by maximum \emph{vertex-weight} matching (MVM).
\stopchange
The running time of Alg.\ \ref{alg:match} increases to $\mathcal{O}\left(n^2|E|\right)$, as MVM
is solved in $\mathcal{O}\left(n|E|\right)$ time \citep{dobrian19}.
An approximate MVM-solution within $\frac{2}{3}$ of optimality is found in $\mathcal{O}\left(|E|+n\log{n}\right)$ time \citep{dobrian19}.
\end{remark}

\section{Experimental Results} \label{sec:exp}
\change

Results from the following numerical experiments are reported in the current section:
\begin{enumerate}
  \item The lower bound $\max_{v\in V\setminus S}sp_v$ \citep{lima22} (see also Observation \ref{obs:loglb}) and the lower bound computed by Alg.\ \ref{alg:dreg}
        are compared. They are also compared with the upper bound found by the fast heuristic method of \citet{harutyunyan14}.
        The Fibonacci lower bound (Prop.\ \ref{prop:lbfib}) is not subject to experiments, since it is dominated by the bound produced by Alg.\ \ref{alg:dreg} (Prop.\ \ref{prop:strength}).
  \item The best lower bound, $lb$, and the upper bound, $ub$, are submitted to both of the ILP approaches
        (direct solution of the model \eqref{mod:basic} of \citet{desousa18,desousa18b} and Alg.\ \ref{alg:ip}, respectively)
        discussed in Section \ref{sec:exact}. A time limit of one hour is imposed on both.
        In the case of Alg.\ \ref{alg:ip}, which runs at most $ub-lb$ iterations, a time limit of $3600s/(ub-lb)$ applies in each iteration. 
        Hence, if the time limit is expired when $t=lb$, while $\nu(t)=n-\sigma$ is observed for $t=lb+1$, then the conclusion $lb\leq\tau(G,S)\leq lb+1$ is drawn.
        Ability to compute the minimum broadcast time, or a smallest possible interval containing it, is reported for both approaches.
  \item Results from the heuristic upper bounding method, Alg.\ \ref{alg:match}, are compared with those produced by the metaheuristic of \citet{lima22}.
        The latter heuristic is parameterised by a \emph{seed}, taking values between 0 and 20.
        One run, subject to a time limit of one minute for each seed value is made, and the best result is recorded.
        Correspondingly, a time limit of 21 minutes is imposed on Alg.\ \ref{alg:match}.
        The tie-breaking rule discussed in Section \ref{sec:cons} is applied.
\end{enumerate}
All experiments are run on a computer with an Intel(R) Core(TM) i5-7500 3.40GHz processor of four cores, each with a single thread.
The computer has 16 GByte RAM memory, and runs Linux (Ubuntu 20.04.5 LTS).
Algorithms \ref{alg:ip} and \ref{alg:match} are implemented in Python 3.10, and the ILP models \eqref{mod:basic} and \eqref{mod:basic:dec} are solved by the Gurobi 9.5.2 solver \citep{gurobi22},
and implemented through the Python interface.
The C++ implementation of the genetic algorithm of \citet{lima22} is downloaded from the authors' git repository.
Other code, that is the upper bounding algorithm of \citet{harutyunyan14} and the lower bounding methods (Observation \ref{obs:loglb}, Alg.\ \ref{alg:dreg}) are implemented in C++.
All C++ code is compiled by version 9.4.0 of the GNU C++ compiler.

\subsection{Instances} \label{sec:instances}
The experiments are run on a set of randomly generated instances, and on all instances studied by \citet{lima22}.
Unlike the latter reference, the current work includes experiments not only on single-source instances.
For each graph under study, a double-source instance is generated by drawing randomly two source nodes.
The graphs belong to \stopchange standard graph classes from the literature \citep[e.g.,][]{graham93},
briefly described in the following paragraphs.

\paragraph{Geometric graph\change s \stopchange on the unit sphere}
The python library \texttt{graph-tool}~\citep{peixoto14} is used to generate geometric graphs
with nodes embedded on the unit sphere in the three-dimensional Euclidean space. 
Two nodes are connected by an edge if the Euclidean distance between them is no larger than a given bound $r$.
The node coordinates are created by normalising three random numbers drawn from a Gaussian distribution.
For $r$ small, the number of connected components in the graph output from \texttt{graph-tool} is $\change m\stopchange>1$.
To ensure connectivity, $\change m\stopchange-1$ additional edges are added arbitrarily such that the resulting graph becomes connected.
 The result is a graph where:
 \begin{itemize}
	 \item If $r$ is sufficiently large, a grid is formed across the unit sphere. This mimics a satellite network, where the edges represent line-of-sight.
	 \item Otherwise, the arising graph is likely to contain local clusters resembling a satellite network.
	 \item For sufficiently small value of $r$, the clusters degenerate to single nodes, and the graph is a tree. 
 \end{itemize}

 \paragraph{Hypercube\change s\stopchange}\label{sect:hc}
 The hypercube graph $Q_d$ is the graph formed from the nodes and edges of a hypercube which is a $d$-dimensional generalisation of a circuit of length four ($d=2$) and a cube ($d=3$).
 Thus, $Q_d$ is a $d$-regular bipartite graph with $2^d$ nodes and $d2^{d-1}$ edges.
With a single source, the minimum broadcast time of the hypercube $Q_d=(V,E)$ is $\tau(Q_d,\{s\})=d$ for all $s\in V$.

 \paragraph{Cube-connected cycle\change s (for brevity, written `CC cycles' whenever convenient)\stopchange}
Consider a graph $G=(V,E)$ and an integer $d\geq 3$, where $|V|=d2^d$ and $E$ defined as follows:
Let the nodes be represented by distinct pairs $(x,y)$ of integers, where $0\leq x<2^d$ and $0\leq y<d$.
Node $(x,y)$ has exactly three neighbours, namely $(x, (y+1) \mod d)$, $(x, (y-1) \mod d)$,
and $(x \oplus 2y, y)$, where $\oplus$ denotes the exclusive or operation on the binary representation of integers.
Thus, $G$ is a cubic graph, referred to as a cube-connected cycle of order $d$ \citep{preparata81}.
It is distinguished from the hypercube $Q_d$ in that each node in $Q_d$ is replaced by a cycle on $d$ nodes, and the edge set is modified such that 3-regularity is obtained,
which in its turn implies $|E|=3d2^{d-1}$.

 \paragraph{Harary graph\change s\stopchange}
\citet{harary62} proves that for all integers $n>k\geq 1$, the minimum edge cardinality of a $k$-connected graph with $n$ nodes is $\lceil\frac{nk}{2}\rceil$.
The same reference provides a procedure that for arbitrary $k$ and $n$ constructs a graph $H_{kn}$, referred to as a Harary graph, at which the minimum is attained.
For instance, $H_{2,n}$ and $H_{n-1,n}$ are, respectively, a circuit and a complete graph, both with $n$ nodes.
The broadcast time in Harary graphs is given particular attention by \citet{bhabak14} and \citet{bhabak17}.

 \paragraph{De Bruijn graph\change s\stopchange}
Each node of a $d$-dimensional De Bruijn graph is represented by a binary string of length $d$.
Two distinct nodes $u$ and $v$ are neighbours if and only if \change the string corresponding to \stopchange $u$
is obtained by shifting all binary digits of \change the string corresponding to \stopchange $v$ one position either left or right,
and either binary symbol is introduced in the vacant position.
Hence, the graph has $2^d$ nodes, each of which has degree at most 4.

 \paragraph{Shuffle exchange graph\change s\stopchange}
Like in De Bruijn graphs, the nodes of a shuffle exchange graph of order $d$ represent binary strings of length $d$. 
There is an edge between two distinct nodes $u$ and $v$ if and only if their corresponding strings are identical in all but their last bit,
or the string corresponding to $u$ is obtained by a left or a right cyclic shift of the bits of $v$.
Hence, the graph has $2^d$ nodes, each of which has degree at most 3.
\change

 \paragraph{Synthetic graphs}
\citet{lima22} have constructed MBT instances for which the minimum broadcast times are known in the single-source cases.
The graphs are designed by adding edges randomly to trees which are known to have broadcast time $\left\lceil\log|V|\right\rceil$ for an appropriate choice of source.
In each such instance, $|V|$ is a power of two, ranging from $2^5$ to $2^{10}$.
 \paragraph{Small world graphs}
The small world graphs included in the experiments consist of 100 or 1000 nodes with average degree ranging from two to six.
All of them are downloaded from the repository of \citet{rossi16}.

\subsection{Lower and upper bounds computed in polynomial time} \label{sec:bounds}
Tables \ref{tab:b-rgg}--\ref{tab:b-SW-1000} in the supplementary material show the node and edge cardinalities (columns 2--3) of all graphs in question.
For the corresponding single-source MBT instances, column 4 contains the lower bounds produced by Alg.\ \ref{alg:dreg},
column 5 contains the lower bounds $\max_{v\in V}sp_v$, and column 6 contains the upper bound found by the method of \citet{harutyunyan14}.
A lower bound is written in bold if it is a strongest lower bound, and an asterisk accompanies all upper bounds that coincide with a corresponding lower bound.
Analogous results for the double-source instances are given in columns 7--9.

\begin{table}\change \caption{\change Lower and upper bounds and their closeness to the best lower bound ($|S|\in\{1,2\}$)\stopchange} \label{tab:bounds}
\centering\begin{tabular}{lrrrcrrrcrrr}
& \multicolumn{2}{c}{Size} & no.\ ins- && \multicolumn{3}{c}{Relative closeness} && \multicolumn{3}{c}{no.\ instances equal} \\ \cline{2-3} \cline{6-8} \cline{10-12}
        Instance set & $|V|$ & $|E|$ & tances & & Alg.\ \ref{alg:dreg} &        $\max$ sp &               ub & & Alg.\ \ref{alg:dreg} &        $\max$ sp &               ub \\\hline
Geometric           &  400& 1220--1816&    8& &0.55 &1.00 &1.13 & &0 (0) &8 (8) &2  \\
Geometric           &  600& 1027--1861&    8& &0.20 &1.00 &1.01 & &0 (0) &8 (8) &7  \\
Geometric           &  800& 1034--1871&    8& &0.07 &1.00 &1.00 & &0 (0) &8 (8) &7  \\
Geometric           & 1000& 1447--2827&    8& &0.14 &1.00 &1.02 & &0 (0) &8 (8) &4  \\
Geometric           & 1200& 1940--4075&    8& &0.18 &1.00 &1.00 & &0 (0) &8 (8) &7  \\
Harary              &   17--100&   17--525&   32& &0.91 &0.75 &1.10 & &24 (19) &13 (8) &16  \\
Hypercube           &   32--1024&   80--5120&   12& &1.00 &0.94 &1.08 & &12 (3) &9 (0) &6  \\
CC cycles           &   24--896&   36--1344&   10& &0.89 &1.00 &1.15 & &2 (0) &10 (8) &0  \\
de Bruijn           &   16--1024&   31--2047&   14& &1.00 &0.91 &1.36 & &14 (10) &4 (0) &0  \\
Shuffle exchange    &   16--1024&   21--1533&   14& &0.83 &0.98 &1.08 & &2 (1) &13 (12) &5  \\
Synthetic           &   32&   31--156&   14& &1.00 &0.69 &1.27 & &14 (12) &2 (0) &2  \\
Synthetic           &   64&   63--558&   14& &1.00 &0.52 &1.34 & &14 (12) &2 (0) &2  \\
Synthetic           &  128&  127--2140&   14& &1.00 &0.47 &1.37 & &14 (12) &2 (0) &2  \\
Synthetic           &  256&  255--8307&   14& &1.00 &0.41 &1.38 & &14 (12) &2 (0) &2  \\
Synthetic           &  512&  511--33313&   14& &1.00 &0.37 &1.41 & &14 (12) &2 (0) &2  \\
Synthetic           & 1024&27259--131643&   12& &1.00 &0.24 &1.44 & &12 (12) &0 (0) &0  \\
Small world         &  100&  100&    6& &0.33 &1.00 &1.00 & &0 (0) &6 (6) &6  \\
Small world         &  100&  200&   36& &0.94 &0.97 &1.34 & &25 (8) &28 (11) &0  \\
Small world         &  100&  300&   18& &1.00 &0.71 &1.27 & &18 (17) &1 (0) &0  \\
Small world         & 1000& 1000&    6& &0.17 &1.00 &1.00 & &0 (0) &6 (6) &6  \\
Small world         & 1000& 2000&   36& &0.92 &0.93 &1.30 & &23 (18) &18 (13) &0  \\
Small world         & 1000& 3000&   18& &1.00 &0.76 &1.36 & &18 (15) &3 (0) &0  \\
\end{tabular}\stopchange
\end{table}

A summary of the results is given for each set of instances in Tab.\ \ref{tab:bounds}.
For the instance set identified by columns 1--3, where column 2 and 3 give the range of node and edge cardinalities, respectively,
column 4 gives the number of instances within the set.
Columns 5--7 give the average score of each lower and upper bound. 
When applied to a particular instance, the score is defined as the bound value divided by \emph{the best lower bound} obtained for that instance.
Thus, a score of a lower bound equal to 1.0 means that it is the best lower bound found, whereas a value smaller than 1.0 implies the converse.
Likewise, the score of the upper bound is 1.0 if the bound coincides with the best lower bound, and greater than 1.0 otherwise.
Closeness to 1.0 of the average score within an instance set thus reflects the strength of the bound when applied to the instances in question.
Columns 8--10 finally show the number of instances in which the respective bounds obtained the score 1.0.
For the lower bounds (columns 8--9), the number of instances in which it is the \emph{unique} bound to obtain this score is given in parentheses.

As could be expected, the tables show that in instances with an eccentric source node, such as the random geometric instances
(rows 1--5 in Tab.\ \ref{tab:bounds}, Tab.\ \ref{tab:b-rgg}) and the small world instances where
$|V|=|E|$ (row -6 and -3 of Tab.\ \ref{tab:bounds}, rows 1--3 of Tabs.\ \ref{tab:b-SW-100-}--\ref{tab:b-SW-1000}), the longest shortest path bound is largely dominant.
The method of \citet{harutyunyan14} is also able to compute an optimal solution in many of these instances, as the provided upper bound coincides with the best lower bound.
In 16 out of 20 (11 out of 20) of the single-source (double-source) random geometric instances (Tab.\ \ref{tab:b-rgg}), for example,
the broadcast time is computed and proved to be minimum uniquely by means of procedures with polynomial running time.
In instances with a more centrally located source node, such as the de Bruijn instances (row 9 of Tab.\ \ref{tab:bounds}, rows $-8,\ldots,-14$ of Tab.\ \ref{tab:b-misc})
and the instances on a rather dense small-world graph (row -1 of Tab.\ \ref{tab:bounds}, rows -1, -4, and -6 of Tab. \ref{tab:b-SW-1000}),
the lower bound of Alg.\ \ref{alg:dreg} dominates both $\left\lceil\log\frac{n}{\sigma}\right\rceil$ and $\max_{v\in V\setminus S}sp_v$.
The upper bounding method, however, fails to close the gap in these instances.

A comparison across all 324 instances shows that the bounds collapse in 42 single-source and 34 double-source instances.
All such instances are classified as \emph{trivial} and will not be pursued in experiments with more time-consuming methods.
Although the majority of the trivial instances have a single source, we find the difference to be too insignificant
to conclude whether double-source instances are generally more challenging than their single-source counterpart.

\subsection{Experiments with ILP approaches and upper-bounding heuristics} \label{sec:ilp}
For all non-trivial instances, Tabs.\ \ref{tab:misc-single}--\ref{tab:SW-1000-double} in the supplementary material show the lower and upper bounds
(optimal solutions if convergence within the time limit)
obtained by the model \eqref{mod:basic} of \citet{desousa18,desousa18b} and Alg.\ \ref{alg:ip}
(columns 2--3 and 4--5, respectively).
The tables also contain the upper bounds obtained by the metaheuristic of \citet{lima22} (column 6),
and the results from Alg.\ \ref{alg:match} with parameter values $\kappa=1,2,3,4$ (columns 7--10, respectively).
Bold-face numbers imply that the bound is no weaker than other bounds reported for the same instance, and an asterisk signifies that an upper bound is no larger than the sharpest lower bound.
A stroke (`--') means that the corresponding method failed to compute the bound in question,
while `\dag' is given to indicate that the solver was interrupted before the time limit because it ran out of memory.

A summary of the results is given in Tabs.\ \ref{tab:ub1}--\ref{tab:ub2}.
Column 2 of both tables gives the number of pursued instances within each set.
For the former ILP approach, columns 3--4 of Tab.\ \ref{tab:ub1} show the computed bounds relative to the lower bound produced by Alg.\ \ref{alg:ip}, averaged over all instances in the set.
Correspondingly, column 5 contains the average value of all upper bounds produced by Alg.\ \ref{alg:ip} relative to the lower bound.
Analogous results for the heuristic methods are given in the last five columns of the table.
Table \ref{tab:ub2} has a column ordering consistent with Tab.\ \ref{tab:ub1},
and shows the number of instances in which the respective bounds are identical to the lower bound produced by Alg.\ \ref{alg:ip}.

\begin{table}\change \caption{\change The table shows the relative closeness within instance sets to the best known lower bound, obtained by ILP approaches and heuristics. Each set consists of instances where $|S|\in\{1,2\}$.} \label{tab:ub1}
\centering\begin{tabular}{lrcrrcrcrcrrrr}
                         &          & &  \multicolumn{4}{c}{ILP approaches} && \multicolumn{6}{c}{Heuristics} \\ \cline{4-7} \cline{9-14}
                         &no.\ ins- & &  \multicolumn{2}{c}{de Sousa} && \multicolumn{1}{c}{Alg\ref{alg:ip}} && \multicolumn{1}{c}{Lima} && \multicolumn{4}{c}{Alg\ref{alg:match}} \\ \cline{4-5} \cline{7-7} \cline{9-9} \cline{11-14}
            Instance set & tances & &               lb &               ub & &               ub & &               ub & &              Ub1 &              Ub2 &              Ub3 &              Ub4 \\\hline
Geometric            & 13  &  & 1.00  & 1.00  &  & 1.00  &  & 1.01  &  & 1.10  & 1.11  & 1.07  & 1.07 \\
Harary               & 16  &  & 1.00  & 1.00  &  & 1.00  &  & 1.00  &  & 1.03  & 1.04  & 1.01  & 1.01 \\
Hypercube            & 6  &  & 1.00  & 1.02  &  & 1.02  &  & 1.04  &  & 1.12  & 1.12  & 1.10  & 1.02 \\
CC cycles            & 10  &  & 1.00  & 1.00  &  & 1.00  &  & 1.00  &  & 1.11  & 1.11  & 1.06  & 1.09 \\
de Bruijn            & 14  &  & 1.00  & 1.01  &  & 1.01  &  & 1.04  &  & 1.13  & 1.12  & 1.09  & 1.05 \\
Shuffle exchange     & 9  &  & 1.00  & 1.00  &  & 1.00  &  & 1.01  &  & 1.11  & 1.11  & 1.10  & 1.12 \\
Synthetic ($|V|=32$) & 12  &  & 1.00  & 1.00  &  & 1.00  &  & 1.00  &  & 1.13  & 1.07  & 1.02  & 1.00 \\
Synthetic ($|V|=64$) & 12  &  & 1.00  & 1.00  &  & 1.00  &  & 1.00  &  & 1.03  & 1.03  & 1.04  & 1.00 \\
Synthetic ($|V|=128$) & 12  &  & 1.00  & 1.00  &  & 1.00  &  & 1.00  &  & 1.01  & 1.03  & 1.00  & 1.00 \\
Synthetic ($|V|=256$) & 12  &  & 1.00  & 1.04  &  & 1.00  &  & 1.00  &  & 1.00  & 1.00  & 1.00  & 1.00 \\
Synthetic ($|V|=512$) & 12  &  & 1.00  & 1.48  &  & 1.02  &  & 1.04  &  & 1.00  & 1.00  & 1.00  & 1.00 \\
Synthetic ($|V|=1024$) & 12  &  & 0.99  & 1.43  &  & 1.43  &  & 1.05  &  & 1.00  & 1.00  & 1.00  & 1.00 \\
Small world ($|V|=10^2$, $|E|=2|V|$) & 36  &  & 1.00  & 1.00  &  & 1.00  &  & 1.00  &  & 1.11  & 1.12  & 1.09  & 1.05 \\
Small world ($|V|=10^2$, $|E|=3|V|$) & 18  &  & 1.00  & 1.00  &  & 1.00  &  & 1.00  &  & 1.09  & 1.11  & 1.05  & 1.03 \\
Small world ($|V|=10^3$, $|E|=2|V|$) & 36  &  & 0.99  & 1.03  &  & 1.03  &  & 1.12  &  & 1.16  & 1.16  & 1.15  & 1.13 \\
Small world ($|V|=10^3$, $|E|=3|V|$) & 18  &  & 0.96  & 1.08  &  & 1.03  &  & 1.13  &  & 1.11  & 1.10  & 1.07  & 1.05 \\
\end{tabular}\stopchange
\end{table}

\begin{table}\change \caption{\change The table shows the number of instances within a set where the best lower bound is met by ILP approaches and heuristics. Each set consists of instances where $|S|\in\{1,2\}$.} \label{tab:ub2}
\centering\begin{tabular}{lrcrrcrcrcrrrr}
                         &          & &  \multicolumn{4}{c}{ILP approaches} && \multicolumn{6}{c}{Heuristics} \\ \cline{4-7} \cline{9-14}
                         &no.\ ins- & &  \multicolumn{2}{c}{de Sousa} && \multicolumn{1}{c}{Alg\ref{alg:ip}} && \multicolumn{1}{c}{Lima} && \multicolumn{4}{c}{Alg\ref{alg:match}} \\ \cline{4-5} \cline{7-7} \cline{9-9} \cline{11-14}
            Instance set & tances & &               lb &               ub & &               ub & &               ub & &              Ub1 &              Ub2 &              Ub3 &              Ub4 \\\hline
Geometric            & 13  &  & 12  & 12  &  & 13  &  & 10  &  & 0  & 0  & 1  & 2 \\
Harary               & 16  &  & 16  & 16  &  & 16  &  & 16  &  & 14  & 13  & 15  & 15 \\
Hypercube            & 6  &  & 6  & 5  &  & 5  &  & 4  &  & 1  & 1  & 2  & 5 \\
CC cycles            & 10  &  & 10  & 10  &  & 10  &  & 10  &  & 2  & 2  & 4  & 3 \\
de Bruijn            & 14  &  & 14  & 12  &  & 13  &  & 9  &  & 3  & 3  & 6  & 8 \\
Shuffle exchange     & 9  &  & 9  & 9  &  & 9  &  & 7  &  & 0  & 1  & 2  & 1 \\
Synthetic ($|V|=32$) & 12  &  & 12  & 12  &  & 12  &  & 12  &  & 5  & 8  & 11  & 12 \\
Synthetic ($|V|=64$) & 12  &  & 12  & 12  &  & 12  &  & 12  &  & 10  & 10  & 9  & 12 \\
Synthetic ($|V|=128$) & 12  &  & 12  & 12  &  & 12  &  & 12  &  & 11  & 10  & 12  & 12 \\
Synthetic ($|V|=256$) & 12  &  & 11  & 11  &  & 12  &  & 12  &  & 12  & 12  & 12  & 12 \\
Synthetic ($|V|=512$) & 12  &  & 3  & 0  &  & 10  &  & 8  &  & 12  & 12  & 12  & 12 \\
Synthetic ($|V|=1024$) & 12  &  & 0  & 0  &  & 0  &  & 7  &  & 12  & 12  & 12  & 12 \\
Small world ($|V|=10^2$, $|E|=2|V|$) & 36  &  & 36  & 36  &  & 36  &  & 36  &  & 10  & 8  & 15  & 22 \\
Small world ($|V|=10^2$, $|E|=3|V|$) & 18  &  & 18  & 18  &  & 18  &  & 18  &  & 9  & 5  & 12  & 14 \\
Small world ($|V|=10^3$, $|E|=2|V|$) & 36  &  & 30  & 24  &  & 26  &  & 3  &  & 0  & 0  & 0  & 1 \\
Small world ($|V|=10^3$, $|E|=3|V|$) & 18  &  & 12  & 8  &  & 13  &  & 0  &  & 3  & 5  & 9  & 12 \\
\end{tabular}\stopchange
\end{table}

A comparison between the model \eqref{mod:basic} of \citet{desousa18,desousa18b} with Alg.\ \ref{alg:ip} in the single-source instance of graph \texttt{SW-1000-6-0d1-trial3}
(see Tab.\ \ref{tab:SW-1000-single}) shows that the former approach gives a better upper bound.
In all other instances, however, Alg.\ \ref{alg:ip} produces lower and upper bounds that are level with or better than those obtained by applying model \eqref{mod:basic}.
Moreover, the algorithm successfully finds the minimum broadcast time and proves its validity in all but 31 instances (217 out of 248 non-trivial instances are solved),
whereas the corresponding success rate of model \eqref{mod:basic} is 197 out of 248.
In their recent research, \citet{lima22} proved optimality in only three out of 30 single-source small-world instances with $|V|=1000$.
By virtue of Alg.\ \ref{alg:ip}, the minimum broadcast time is now known in 19 more of these instances (see Tab.\ \ref{tab:SW-1000-single}).

\stopchange

Let $\text{ub}(\alpha,\beta)$ denote the upper bound output by method $\alpha$ when applied to instance $\beta$,
and let $\text{ub}^{\min}(\beta)$ and $\text{ub}^{\max}(\beta)$ denote, respectively, the corresponding minimum and maximum values taken over all methods $\alpha$.
The \emph{performance profile} of method $\alpha$ is defined as the function $\varphi:[0,1]\to[0,1]$, where $\varphi(x)$ equals the proportion of instances $\beta$ in which 
\[
  \text{ub}(\alpha,\beta) - \text{ub}^{\min}(\beta) \leq \left(\text{ub}^{\max}(\beta) - \text{ub}^{\min}(\beta)\right)x.
\]

\begin{figure}
\begin{center}
\begin{tabular}{cc}
\includegraphics[scale=0.5]{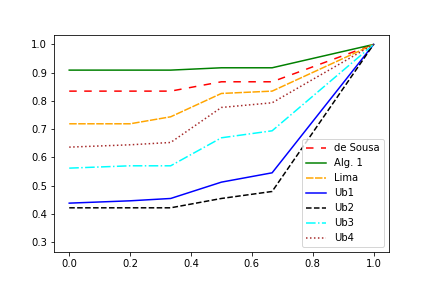} &
\includegraphics[scale=0.5]{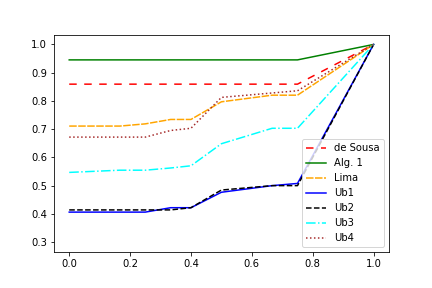} \\
\scriptsize{\textbf{One source}} & \scriptsize{\textbf{Two sources}}
\end{tabular}
\caption{\change Performance profiles of the upper bounding methods\stopchange} \label{fig:ubperf}
\end{center}
\end{figure}

Figure \ref{fig:ubperf} summarise all experiments reported in Tabs.\ \ref{tab:ub1}--\ref{tab:ub2} in terms of performance profiles.
Two separate sets of profiles are given for the cases $\sigma=1$ and $\sigma=2$ for comparison of \change all upper bounding methods (Fig.\ \ref{fig:ubperf}),
including the two time-constrained ILP approaches.

The dominance of Alg.\ \ref{alg:ip} (profile `Alg.\ 1') over the model of \citet{desousa18,desousa18b} (profile `de Sousa') is highly visible in Fig.\ \ref{fig:ubperf}.
Reflecting the fact that Alg.\ \ref{alg:ip} solves most of the instances to optimality, and provides the best upper bound in most of the remaining instances,
the ordinate values of the left-most points of the corresponding performance profiles are larger than 90\% and 95\% for the single-source and double-source experiments, respectively.

A comparison of the heuristic methods shows that in the single-source instances,
the genetic algorithm of \citet{lima22} (profile `Lima') performs better than Alg.\ \ref{alg:match} (profile `Ub$\kappa$') for all $\kappa=1,\ldots,4$.
For $\sigma=2$, however, this is true only when $\kappa\leq 3$, and Alg.\ \ref{alg:match} becomes competitive when $\kappa=4$.
The favourable performance of the genetic algorithm is also mainly explained by better results in the smaller instances.
Figure \ref{fig:large} depicts the performance profiles confined to the instances in which $|V|\geq 1000$, including both $\sigma=1$ and $\sigma=2$.
Among the instances excluded by this criterion, Alg.\ \ref{alg:ip} solves to optimality all but two single-source instances,
which justifies the focus to the restricted instance set.

\begin{figure}[ht!]
\begin{center}
\includegraphics[scale=0.7]{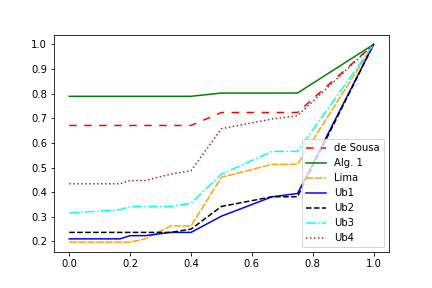}
\caption{\change Performance profiles of the upper bounding methods in instances where $|V|\geq 1000$\stopchange} \label{fig:large}
\end{center}
\end{figure}

\noindent
It is observed from Fig.\ \ref{fig:large} that Alg.\ \ref{alg:match} performs better than the genetic algorithm \citep{lima22}, provided that $\kappa\geq 3$.
For $\kappa=3$, the difference is modest, whereas it becomes significant for $\kappa=4$.
Figures \ref{fig:ubperf}--\ref{fig:large} also show that there is no added value of increasing the value of $\kappa$ from 1 to 2 in Alg.\ \ref{alg:match}.

\subsection{Solution time} \label{sec:time}

Tables \ref{tab:time-misc}--\ref{tab:time-SW} report solution times in seconds for all but one of the methods analysed in Section \ref{sec:ilp} in all instances of some computational challenge.
Since in the majority of the instances, the method of \citet{lima22} continues the search as long as the given time limit (60s for each of 21 seed values) is not reached,
it is excluded from the solution time analysis.
In an order consistent with Tabs.\ \ref{tab:misc-single}--\ref{tab:SW-1000-double}, columns 2--7 (columns 8--13) contain the running times for single-source (double-source) instances.
Arguing that running time is unlikely to be an issue in instances for which optimality is provable in a one-digit number of seconds,
we include only instances in which at least one method needs 10 seconds or more to conclude.
A stroke (`--') is given for trivial instances (see Section \ref{sec:bounds}), and, in line with Tabs.\ \ref{tab:misc-single}--\ref{tab:SW-1000-double},
the symbol `\dag' corresponds to runs interrupted by memory shortage.
For each instance set, average running times are given in Tab.\ \ref{tab:t}, where runs exhausting the memory are considered to take 3600 seconds.
\begin{table}[ht!]\change \caption{\change Running times (seconds) averaged over all instances ($|S|\in\{1,2\}$) in each set\stopchange } \label{tab:t}
\centering\begin{tabular}{lrrcrrrr}
                     & \multicolumn{2}{c}{ILP approaches} && \multicolumn{4}{c}{Heuristics} \\\cline{2-3} \cline{5-8}
Instance set         & \multicolumn{1}{c}{de Sousa} & \multicolumn{1}{c}{Alg\ref{alg:ip}} && Ub1& Ub2& Ub3& Ub4\\\hline
Geometric            &  831.7 &  162.0 &&    1.0 &    1.6 &    2.3 &    3.2 \\
Harary               &    0.4 &    0.1 &&    0.0 &    0.0 &    0.0 &    0.0 \\
Hypercube            &  675.5 &  696.4 &&    0.1 &    0.2 &    0.5 &    4.4 \\
CC cycles            &    4.2 &    7.7 &&    0.1 &    0.1 &    0.2 &    0.2 \\
de Bruijn            &  792.3 &  102.0 &&    0.1 &    0.1 &    0.2 &    0.3 \\
Shuffle exchange     &   42.1 &   17.4 &&    0.1 &    0.2 &    0.3 &    0.4 \\
Synthetic ($|V|=32$) &    0.2 &    0.1 &&    0.0 &    0.0 &    0.0 &    0.0 \\
Synthetic ($|V|=64$) &    1.1 &    0.2 &&    0.0 &    0.0 &    0.1 &    0.1 \\
Synthetic ($|V|=128$) &   49.7 &    1.9 &&    0.1 &    0.1 &    0.3 &    0.5 \\
Synthetic ($|V|=256$) & 1187.7 &   36.3 &&    0.2 &    0.6 &    1.6 &    3.6 \\
Synthetic ($|V|=512$) & 3607.2 &  573.5 &&    0.6 &    3.2 &   11.2 &   29.2 \\
Synthetic ($|V|=1024$) & 3600.0 & 3611.2 &&    2.4 &   18.0 &   81.2 &  287.5 \\
Small world ($|V|=10^2$, $|E|=2|V|$) &    2.6 &    0.4 &&    0.0 &    0.0 &    0.1 &    0.1 \\
Small world ($|V|=10^2$, $|E|=3|V|$) &    2.7 &    0.5 &&    0.0 &    0.0 &    0.1 &    0.2 \\
Small world ($|V|=10^3$, $|E|=2|V|$) & 2304.7 &  595.9 &&    0.2 &    0.4 &    0.8 &    3.3 \\
Small world ($|V|=10^3$, $|E|=3|V|$) & 2596.5 &  761.0 &&    0.3 &    0.4 &    0.9 &   13.6 \\
\end{tabular}\stopchange
\end{table}

Computational superiority of Alg.\ \ref{alg:ip} over model \eqref{mod:basic} is confirmed by the running times.
The latter approach is, however, faster in 10 instances.
The most significant difference in its favour is found in the double-source instance of graph \texttt{rgg-1000-2792},
in which Alg.\ \ref{alg:ip} needed almost seven times the running time of \eqref{mod:basic}.
Other instances that are exceptions to the general rule, are the graphs \texttt{cubeconnectedcycles7} ($\sigma=1,2$), \texttt{shuffle\_exchange10} ($\sigma=1$),
\texttt{rgg-1200-3855} ($\sigma=2$),
\texttt{hypercube8} ($\sigma=2$),
\texttt{hypercube9} ($\sigma=2$),
\texttt{SW-1000-4-0d3-trial2} ($\sigma=1$),
\texttt{SW-1000-6-0d2-trial3} ($\sigma=1$),
\texttt{SW-1000-5-0d1-trial3} ($\sigma=2$),
and \texttt{SW-1000-6-0d2-trial3} ($\sigma=2$).
But in 52 of the instances that both could solve to optimality, Alg.\ \ref{alg:ip} spent less than half the time the solver needed to solve the model of \citet{desousa18,desousa18b}.
A graphic illustration is given in Fig.\ \ref{fig:time}, which shows the running time of Alg.\ \ref{alg:ip} versus the model of \citet{desousa18,desousa18b} in all said instances.

\begin{figure}[ht!]
\begin{center}
\begin{tabular}{cc}
\includegraphics[scale=0.5]{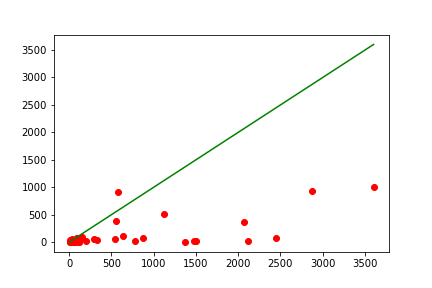} &
\includegraphics[scale=0.5]{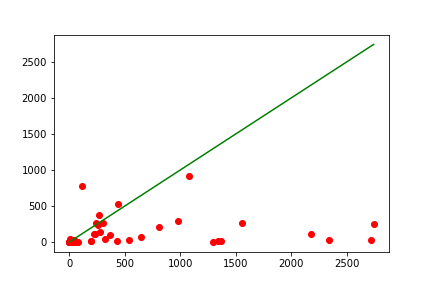} \\
\scriptsize{\textbf{One source}} & \scriptsize{\textbf{Two sources}}
\end{tabular}
\caption{\change Running times (seconds) of Alg.\ \ref{alg:ip} (vertical axis) vs running times of model \eqref{mod:basic} (horizontal axis)\stopchange} \label{fig:time}
\end{center}
\end{figure}

As expected, the running time of heuristic Alg.\ \ref{alg:match} increases with increasing value of the parameter $\kappa$.
In all small world instances but one, however, and in all other instances except five (six) of the more challenging single-source (double-source) instances of synthetic graphs,
the running time is kept below two minutes, even for $\kappa=4$.

\begin{figure}[ht!]
\begin{center}
\begin{tabular}{cc}
\includegraphics[scale=0.5]{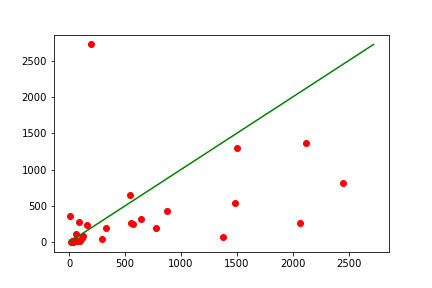} &
\includegraphics[scale=0.5]{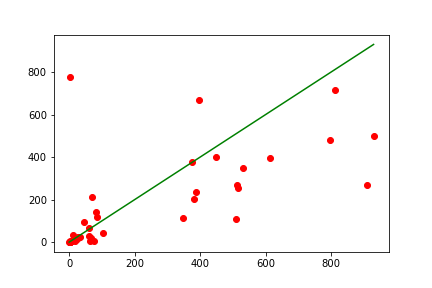} \\
\scriptsize{\textbf{Model \eqref{mod:basic}\citep{desousa18,desousa18b}}} &
\scriptsize{\textbf{Alg.\ \ref{alg:ip}}}
\end{tabular}
\caption{\change Running times (seconds) of the ILP approaches applied to double-source instances (vertical axis) vs single-source instances (horizontal axis)\stopchange} \label{fig:svsd}
\end{center}
\end{figure}

When comparing the single-source and the double-source instances corresponding to the same graph, the experiments give no conclusive evidence that either source cardinality is more or less challenging.
For both ILP approaches under consideration, Fig.\ \ref{fig:svsd} plots the running times of the double-source instances against the running time of its single-source counterpart.
This is done for all graphs where the ILP approach was able to solve both instances to optimality within the time limit.
Visual inspection suggests a bias towards the conclusion that the single-source instances are somewhat more challenging.
Algorithm \ref{alg:ip} fails to prove optimality in 16 single-source and 15 double-source instances.
Out of 40 graphs for which both instances are non-trivial, and the algorithm solves both to optimality, and needs at least 10 seconds to do so,
the single-source (double-source) instance is solved faster for 12 (28) graphs.
\stopchange

\section{Concluding Remarks} \label{sec:conc}

This work focuses on the minimum broadcast time problem, and presents several techniques for computing lower bounds, upper bounds, as well as optimal solutions.
Particular attention is given to a procedure which in each iteration solves an integer linear programming model.
When run exhaustively, this procedure solves the problem. Otherwise, it computes a lower bound on the broadcast time.
The same procedure applied to the continuous relaxation of the model is also capable of computing a lower bound.
Further, an upper-bounding iterative technique is studied.
This method solves a sequence of subproblems, each of which is a possibly small instance of the integer program.
With its parameter decisive for the size of the subproblem instances, the upper bounding method offers high flexibility in the trade-off between sharpness of the bound and computational effort.

For experimental evaluation of the computational procedures,
various instance classes of variable size are addressed.
While most instance sets are \change identical to those studied in a recently published work \stopchange on the same problem, also new, randomly generated instances are studied. 
The random instances are intended to simulate real communication networks.

Computational experiments demonstrate that \change the majority of the instances that cannot be solved by fast bound-computing algorithms, are solved 
by the procedure generating a sequence of integer linear programs.
When interrupted because the time limit is reached, the procedure produces bounds that are generally stronger than those produced within the same time limit by a previously studied ILP model.
In such instances, where the exact approach fails to prove optimality, the heuristic developed in the current work outputs solutions superior to those produced by a recently studied metaheuristic,
and does so with modest computational effort. \stopchange

There is a potential for future research in developing stronger upper bounding algorithms and improving the existing ILP model.
Although the model formulation is compact, its size represents a challenge due to a cubic number of variables.
Model improvements can be achieved by not only introduction of redundant valid inequalities,
but also by developing conceptually different models, where the number of variables is reduced by an order of magnitude.

\nocite{*}

\bibliographystyle{itor}
\bibliography{itor}

\begin{thebibliography}{}
%
%
\bibitem[Bar-Noy et~al.(2000)]{barnoy00}
Bar-Noy, A., Guha, S., Naor, J., Schieber, B., 2000.
Multicasting in heterogeneous networks.
\emph{SIAM Journal on Computing} 30, 2, 347--358.

\bibitem[Bermond et~al.(1998)]{bermond98}
Bermond, J., Gargano, L., Perennes, S., 1998.
Optimal sequential gossiping by short messages.
\emph{Discrete Applied Mathematics} 86, 145--155.

\bibitem[Bermond et~al.(1995)]{bermond95}
Bermond, J., Gargano, L., Rescigno, A.A., Vaccaro, U., 1995.
Fast gossiping by short messages.
\emph{Lecture Notes in Computer Science} 944, 135--146.

\bibitem[Bhabak et~al.(2014)]{bhabak14}
Bhabak, P., Harutyunyan, H.A., Tanna, S., 2014.
Broadcasting in Harary-Like Graphs.
\emph{IEEE 17th International Conference on Computational Science and Engineering (CSE 2014)}, pp.\ 1269--1276.

\bibitem[Bhabak et~al.(2017)]{bhabak17}
Bhabak, P., Harutyunyan, H.A., Kropf, P.G., 2017.
Efficient Broadcasting Algorithm in Harary-like Networks.
\emph{46th International Conference on Parallel Processing Workshop (ICPP)}, pp.\ 162--170. 

\bibitem[Bucantanschi et al.(2007)]{bucantanschi07}
Bucantanschi, D., Hoﬀmann, B., Hutson, K.R., Kretchmar, R.M., 2007.
A neighborhood search technique for the freeze tag problem.
In \emph{Extending the Horizons: Advances in Computing, Optimization, and Decision Technologies}.
Springer, Berlin, pp. 97--113.

\bibitem[Chu and Chen(2017)]{chu17}
Chu, X., Chen, Y., 2017.
Time division inter-satellite link topology generation problem: Modeling and solution.
\emph{International Journal of Satellite Communications and Networking} 36, 194--206.

\bibitem[Cplex(2020)]{cplex20}
Cplex, 20.1., V20.1.0.0: User's Manual for CPLEX,
International Business Machines Corporation, 2020.


\bibitem[Dekker(2002)]{dekker02}
Dekker, A., 2002.
Applying social network analysis concepts to military c4isr architectures.
\emph{Connections}, 24, 3, 93--103.

\bibitem[Dobrian et~al.(2019)]{dobrian19}
Dobrian, F., Halappanavar, M., Pothen, A., Al--Herz, A., 2019.
A 2/3-approximation algorithm for vertex weighted matching in bipartite graphs.
\emph{SIAM Journal of Scientific Computing} 41, 1, A566--A591.

\bibitem[Elkin and Kortsarz(2003)]{elkin03}
Elkin, M., Kortsarz, G., 2003.
Sublogarithmic approximation for telephone multicast: path out of jungle.
\emph{Symposium on Discrete Algorithms}, pp.\ 76--85.



\bibitem[Fraigniaud and Lazard(1994)]{fraigniaud94}
Fraigniaud, P., Lazard, E., 1994.
Methods and problems of communication in usual networks.
\emph{Discrete Applied Mathematics} 53, 1--3, 79--133.

\bibitem[Garey and Johnson(1979)]{garey79}
Garey, M.R., Johnson, D.S., 1979.
\emph{Computers and Intractability: A Guide to the Theory of NP-Completeness},
W.H. Freeman and Co, San Francisco, California, USA, 1979.

\bibitem[Graham  and Harary(1993)]{graham93}
Graham, N., Harary, F., 1993.
Hypercubes, shuffle-exchange graphs and de Bruijn digraphs.
\emph{Mathematical and Computer Modelling} 17, 11, 69--74.

\bibitem[Grigni and Peleg(1991)]{grigni91}
Grigni, M., Peleg, D., 1991.
Tight bounds on minimum broadcast networks.
\emph{Networks} 4, 207--222.

\bibitem[Gurobi Optimization(2022)]{gurobi22}
Gurobi Optimization, LLC, 2022,
Gurobi Optimizer Reference Manual, \href{https://www.gurobi.com}{https://www.gurobi.com}.

\bibitem[Harary(1962)]{harary62}
Harary, F., 1962.
Maximum connectivity of a graph.
\emph{Proceedings of the National Academy of Sciences of the United States of America} 48, 7, 1142--1145.

\bibitem[Harutyunyan et~al.(2013)]{harutyunyan13}
Harutyunyan, H.A., Liestman, A.L., Peters, J.G., Richards, D., 2013.
Broadcasting and gossiping. In Gross, J., Yellen, J., Zhang, P.\ (eds): \emph{Handbook of Graph Theory},
Chapman and Hall, Boca Raton, Florida, 2013, pp.\ 1477--1494.

\bibitem[Harutyunyan and Jimborean(2014)]{harutyunyan14}
Harutyunyan, H.A., Jimborean, C., 2014.
New Heuristic for Message Broadcasting in Network.
\emph{IEEE 28th International Conference on Advanced Information Networking and Application}, pp.\ 517--524.

\bibitem[Harutyunyan and Shao(2006)]{harutyunyan06}
Harutyunyan, H.A., Shao, B., 2006.
An efficient heuristic for broadcasting in networks.
\emph{Journal of Parallel and Distributed Computing} 66, 1, 68--76.

\bibitem[Harutyunyan and Wang(2010)]{harutyunyan10}
Harutyunyan, H.A., Wang, W, 2010.
Broadcasting algorithm via shortest paths.
\emph{16th International Conference on Parallel and Distributed Systems (ICPADS)}, pp.\ 299--305.

\bibitem[Hasson  and Sipper(2004)]{hasson04} 
Hasson, Y., Sipper, M., 2004.
A Novel Ant Algorithm for Solving the Minimum Broadcast Time Problem.
\emph{International Conference on Parallel Problem Solving from Nature}, pp.\ 775--780.

\bibitem[Hedetniemi et~al.(1988)]{hedetniemi88}
Hedetniemi, S.M., Hedetniemi, S.T., Liestman, A.L., 1988.
A survey of gossiping and broadcasting in communication networks.
\emph{Networks}, 18, 4, 319--349.

\bibitem[Hopcroft and Karp(1973)]{hopcroft73}
Hopcroft, J.E., Karp, R.M., 1973.
An $n^{5/2}$ algorithm for maximum matchings in bipartite graphs.
\emph{SIAM Journal on Computing} 2, 4, 225--23.

\bibitem[Hocao\v{g}lu and Gen\c{c}(2019)]{hocaoglu19}
Hocao\v{g}lu M.F., Gen\c{c}, \.{I}, 2019.
Smart combat simulations in terms of industry 4.0.
In \emph{Simulation for Industry 4.0}, Springer, Berlin, 247--273.

\bibitem[Hromkovi\v{c} et~al.(1996)]{hromkovic96}
Hromkovi\v{c}, J., Klasing, R., Monien, B, Peine, R., 1996.
Dissemination of Information in Interconnection Networks (Broadcasting \& Gossiping).
In: Du, D.Z., Hsu, D.F.\ (eds) \emph{Combinatorial Network Theory, Applied Optimization}, Springer, Boston, Massachusetts, vol 1, pp.\ 125--212.


\bibitem[Jansen and M\"uller(1995)]{jansen95}
Jansen, K., M\"uller, H., 1995.
The minimum broadcast time problem for several processor networks.
\emph{Theoretical Computer Science} 147, 69--85.


\bibitem[Kortsarz and Peleg(1995)]{kortsarz95}
Kortsarz, G., Peleg, D., 1995.
Approximation algorithms for minimum-time broadcast.
\emph{SIAM Journal on Discrete Mathematics} 8, 3, 401--427.

\bibitem[Lazard(1992)]{lazard92}
Lazard, E., 1992.
Broadcasting in dma-bound bounded degree graphs.
\emph{Discrete Applied Mathematics}, 37-38, 387--400.

\bibitem[Lima et~al.(2022)]{lima22}
Lima, A., Aquino, A.L.L, Nogueira, B., Pinheiro, R.G.S., 2022.
A matheuristic approach for the minimum broadcast time problem using a biased random-key genetic algorithm.
\emph{International Transactions in Operational Research}, to appear, 2022.

\bibitem[McGarvey et~al.(2016)]{mcgarvey16}
McGarvey, R.G., Rieksts, B.Q., Ventura, J.A., Ahn,N., 2016.
Binary linear programming models for robust broadcasting in communication networks.
\emph{Discrete Applied Mathematics} 204, 173--84.

\bibitem[Middendorf(1993)]{middendorf93}
Middendorf, M., 1993.
Minimum broadcast time is NP-complete for 3-regular planar graphs and deadline 2.
\emph{Information Processing Letters} 46, 281--287.

\bibitem[Noe and Post(2005)]{noe05}
Noe, T. D., Post, J. V., 2005.
Primes in Fibonacci n-step and Lucas n-Step Sequences.
\emph{J. Integer Seq.} 8, Article 05.4.4.

\bibitem[Peixoto(2014)]{peixoto14}
Peixoto, T.P., 2014.
The graph-tool python library.
10.6084/M9.FIGSHARE.1164194.V13.

\bibitem[Preparata and Vuillemin(1981)]{preparata81}
Preparata, F.P., Vuillemin, J., 1981.
The Cube-Connected Cycles: A Versatile Network for Parallel Computation.
\emph{Computer Architecture and Systems} 24, 5, 300--309.

\bibitem[Proskurowski (1981)]{proskurowski81}
Proskurowski, A., 1981.
Minimum Broadcast Trees.
\emph{IEEE Transactions on Computers} C-30, 5, 363--366.

\bibitem[Ravi(1994)]{ravi94}
Ravi, R., 1994.
Rapid rumor ramification: Approximating the minimum broadcast time.
\emph{Proceedings 35th Annual Symposium on Foundations of Computer Science}, Santa Fe, New Mexico, pp.\ 202-213.

\bibitem[Rossi and Ahmed (2016)]{rossi16}
Rossi, R.A, Ahmed, N.K, 2016.
An Interactive Data Repository with Visual Analytics,
\emph{SIGKDD Explor.} 17, 2, 37--41.

\bibitem[Scheuermann and Wu(1984)]{scheuermann84}
Scheuermann, P., Wu, G., 1984.
Heuristic Algorithms for Broadcasting in Point-to-Point Computer Networks.
\emph{IEEE Transactions on Computers} 33, 9, 804--811.

\bibitem[Shang et~al.(2010)]{shang10}
Shang, W., Wan, P., Hu, X., 2010.
Approximation algorithms for minimum broadcast schedule problem in wireless sensor networks.
\emph{Frontiers of Mathematics in China} 5, 1, 75--87.

\bibitem[Slater et~al.(1981)]{slater81}
Slater, P. J., Cockayne, E. J., Hedetniemi, S.T., 1981.
Information dissemination in Trees.
\emph{SIAM Journal on Computing} 10, 4, 692--701.

\bibitem[de Sousa et~al.(2018a)]{desousa18}
de Sousa, A., Gallo, G., Gutierrez, S., Robledo, F., Rodr\'{i}guez-Bocca, P., Romero, P., 2018.
Heuristics for the minimum broadcast time.
\emph{Electronic Notes in Discrete Mathematics} 69, 165--172.

\bibitem[de Sousa et~al.(2018b)]{desousa18b}
de Sousa, A., Robledo, F., Rodr\'{i}guez-Bocca, P., Romero, P., Gallo, G., Gutierrez, S., 2018.
Heuristics for the minimum broadcast time.
Unpublished report,
retrieved from https://www.fing.edu.uy/\~{}frobledo/Paper\_MBT\_2018.pdf

\bibitem[Wanf(2010)]{wang10}
Wang, W., 2010.
Heuristics for Message Broadcasting in Arbitrary Networks.
Master thesis, Concordia University, Montr\'eal, Qu\'ebec, 
retrieved from http://citeseerx.ist.psu.edu/viewdoc/download?doi=10.1.1.633.5827\&rep=rep1\&type=pdf

\end{thebibliography}


\clearpage
\begin{appendices}
\begin{center}
{\Large SUPPLEMENTARY MATERIAL}
\end{center}

\section{Lower and upper bounds} \label{ap:table-b}
\begin{table}[ht!]\change \caption{\change Lower and upper bounds - geometric graphs on the unit sphere\stopchange } \label{tab:b-rgg}
\centering
\begin{tabular}{lrrcrrrcrrr}
                         & \multicolumn{2}{c}{Size} && \multicolumn{3}{c}{$|S|=1$} && \multicolumn{3}{c}{$|S|=2$} \\ \cline{2-3} \cline{5-7} \cline{9-11}
                Instance & $|V|$ & $|E|$ & &           degree &        $\max$ sp &               ub & &           degree &        $\max$ sp &               ub \\\hline
rgg-400-1220            &  400& 1220& &                  9 &        \textbf{ 19}&                 21         & &                  8 &        \textbf{ 18}&                 19          \\
rgg-400-1264            &  400& 1264& &                  9 &        \textbf{ 23}&                 23$^{\ast}$& &                  8 &        \textbf{ 23}&                 23$^{\ast}$ \\
rgg-400-1779            &  400& 1779& &                  9 &        \textbf{ 13}&                 17         & &                  8 &        \textbf{ 11}&                 14          \\
rgg-400-1816            &  400& 1816& &                  9 &        \textbf{ 14}&                 16         & &                  8 &        \textbf{ 12}&                 14          \\
rgg-600-1027            &  600& 1027& &                 10 &        \textbf{208}&                208$^{\ast}$& &                  9 &        \textbf{140}&                140$^{\ast}$ \\
rgg-600-1087            &  600& 1087& &                 10 &        \textbf{255}&                255$^{\ast}$& &                  9 &        \textbf{110}&                110$^{\ast}$ \\
\end{tabular}\stopchange
\end{table}

\begin{table}[ht!]\change \caption{\change Lower and upper bounds - miscellaneous graphs\stopchange } \label{tab:b-misc}
\centering\begin{tabular}{lrrcrrrcrrr}
                         & \multicolumn{2}{c}{Size} && \multicolumn{3}{c}{$|S|=1$} && \multicolumn{3}{c}{$|S|=2$} \\ \cline{2-3} \cline{5-7} \cline{9-11}
                Instance & $|V|$ & $|E|$ & &           degree &        $\max$ sp &               ub & &           degree &        $\max$ sp &               ub \\\hline
harary17c2              &   17&   17& &        \textbf{  9}&                  8 &                  9$^{\ast}$& &                  5 &        \textbf{  7}&                  7$^{\ast}$ \\
harary17c3              &   17&   26& &        \textbf{  5}&        \textbf{  5}&                  6         & &        \textbf{  4}&        \textbf{  4}&                  5          \\
harary17c5              &   17&   43& &        \textbf{  5}&                  3 &                  6         & &        \textbf{  4}&                  2 &                  4$^{\ast}$ \\
harary17c6              &   17&   51& &        \textbf{  5}&                  3 &                  5$^{\ast}$& &        \textbf{  4}&                  2 &                  4$^{\ast}$ \\
harary17c7              &   17&   60& &        \textbf{  5}&                  2 &                  5$^{\ast}$& &        \textbf{  4}&                  2 &                  4$^{\ast}$ \\
harary30c2              &   30&   30& &        \textbf{ 15}&        \textbf{ 15}&                 15$^{\ast}$& &                  8 &        \textbf{ 11}&                 11$^{\ast}$ \\
\end{tabular}\stopchange
\end{table}

\begin{table}\change \caption{\change Lower and upper bounds - synthetic graphs\stopchange } \label{tab:b-BT}
	\centering
\begin{tabular}{lrrcrrrcrrr}
                         & \multicolumn{2}{c}{Size} && \multicolumn{3}{c}{$|S|=1$} && \multicolumn{3}{c}{$|S|=2$} \\ \cline{2-3} \cline{5-7} \cline{9-11}
                Instance & $|V|$ & $|E|$ & &           degree &        $\max$ sp &               ub & &           degree &        $\max$ sp &               ub \\\hline
BT4                     &   16&   15& &        \textbf{  4}&        \textbf{  4}&                  4$^{\ast}$& &        \textbf{  4}&        \textbf{  4}&                  4$^{\ast}$ \\
BT5                     &   32&   31& &        \textbf{  5}&        \textbf{  5}&                  5$^{\ast}$& &        \textbf{  5}&                  4 &                  5$^{\ast}$ \\
BT6                     &   64&   63& &        \textbf{  6}&        \textbf{  6}&                  6$^{\ast}$& &        \textbf{  6}&        \textbf{  6}&                  6$^{\ast}$ \\
BT7                     &  128&  127& &        \textbf{  7}&        \textbf{  7}&                  7$^{\ast}$& &        \textbf{  6}&        \textbf{  6}&                  6$^{\ast}$ \\
BT8                     &  256&  255& &        \textbf{  8}&        \textbf{  8}&                  8$^{\ast}$& &        \textbf{  8}&        \textbf{  8}&                  8$^{\ast}$ \\
BT9                     &  512&  511& &        \textbf{  9}&        \textbf{  9}&                  9$^{\ast}$& &        \textbf{  9}&        \textbf{  9}&                  9$^{\ast}$ \\
\end{tabular}\stopchange
\end{table}

\begin{table}[ht!]\change \caption{\change Lower and upper bounds - small world ($|V|=100$) graphs\stopchange } \label{tab:b-SW-100-}
\centering\begin{tabular}{lrrcrrrcrrr}
                         & \multicolumn{2}{c}{Size} && \multicolumn{3}{c}{$|S|=1$} && \multicolumn{3}{c}{$|S|=2$} \\ \cline{2-3} \cline{5-7} \cline{9-11}
                Instance & $|V|$ & $|E|$ & &           degree &        $\max$ sp &               ub & &           degree &        $\max$ sp &               ub \\\hline
SW-100-3-0d1-trial1     &  100&  100& &                 11 &        \textbf{ 61}&                 61$^{\ast}$& &                  9 &        \textbf{ 22}&                 22$^{\ast}$ \\
SW-100-3-0d2-trial1     &  100&  100& &                 12 &        \textbf{ 31}&                 31$^{\ast}$& &                 10 &        \textbf{ 31}&                 31$^{\ast}$ \\
SW-100-3-0d2-trial3     &  100&  100& &                 12 &        \textbf{ 31}&                 31$^{\ast}$& &                 10 &        \textbf{ 31}&                 31$^{\ast}$ \\
SW-100-4-0d1-trial1     &  100&  200& &        \textbf{  7}&        \textbf{  7}&                 11         & &                  6 &        \textbf{  7}&                  9          \\
SW-100-4-0d1-trial2     &  100&  200& &        \textbf{  7}&        \textbf{  7}&                 10         & &        \textbf{  6}&                  5 &                  8          \\
SW-100-4-0d1-trial3     &  100&  200& &                  7 &        \textbf{  9}&                 11         & &                  6 &        \textbf{  7}&                  9          \\
\end{tabular}\stopchange
\end{table}

\begin{table}[ht!]\change \caption{\change Lower and upper bounds - small world ($|V|=1000$) graphs\stopchange } \label{tab:b-SW-1000}
\centering\begin{tabular}{lrrcrrrcrrr}
                         & \multicolumn{2}{c}{Size} && \multicolumn{3}{c}{$|S|=1$} && \multicolumn{3}{c}{$|S|=2$} \\ \cline{2-3} \cline{5-7} \cline{9-11}
                Instance & $|V|$ & $|E|$ & &           degree &        $\max$ sp &               ub & &           degree &        $\max$ sp &               ub \\\hline
SW-1000-3-0d2-trial1    & 1000& 1000& &                 14 &        \textbf{ 89}&                 89$^{\ast}$& &                 13 &        \textbf{ 89}&                 89$^{\ast}$ \\
SW-1000-3-0d2-trial2    & 1000& 1000& &                 14 &        \textbf{ 88}&                 88$^{\ast}$& &                 13 &        \textbf{ 81}&                 81$^{\ast}$ \\
SW-1000-3-0d3-trial2    & 1000& 1000& &                 13 &        \textbf{ 87}&                 87$^{\ast}$& &                 12 &        \textbf{ 52}&                 52$^{\ast}$ \\
SW-1000-4-0d1-trial1    & 1000& 2000& &                 11 &        \textbf{ 14}&                 18         & &                 10 &        \textbf{ 12}&                 15          \\
SW-1000-4-0d1-trial2    & 1000& 2000& &                 11 &        \textbf{ 15}&                 18         & &                 10 &        \textbf{ 14}&                 16          \\
SW-1000-4-0d1-trial3    & 1000& 2000& &                 11 &        \textbf{ 15}&                 18         & &                 10 &        \textbf{ 13}&                 16          \\
\end{tabular}\stopchange
\end{table}

\section{ILP approaches and heuristics} \label{ap:table-s}
\begin{table}[ht!]\change \caption{\change Results from experiments with ILP approaches and heuristics - miscellaneous single-source instances\stopchange} \label{tab:misc-single}
\centering\begin{tabular}{lcccccccccccc}
                         & \multicolumn{5}{c}{ILP approaches} && \multicolumn{6}{c}{Heuristics} \\ \cline{2-6} \cline{8-13}
                         & \multicolumn{2}{c}{de Sousa} && \multicolumn{2}{c}{Alg\ref{alg:ip}} && \multicolumn{1}{c}{ } && \multicolumn{4}{c}{Alg\ref{alg:match}} \\ \cline{2-3} \cline{5-6} \cline{10-13}
                Instance &               lb &               ub & &               lb &               ub & &             Lima & &              Ub1 &              Ub2 &              Ub3 &              Ub4 \\\hline
rgg-400-1220             & \textbf{              20}& \textbf{              20}$^{\ast}$& & \textbf{              20}& \textbf{              20}$^{\ast}$& & \textbf{              20}$^{\ast}$& &               22 &               22 &               22 &               22 \\
rgg-400-1779             & \textbf{              15}& \textbf{              15}$^{\ast}$& & \textbf{              15}& \textbf{              15}$^{\ast}$& & \textbf{              15}$^{\ast}$& &               17 &               17 &               16 &               16 \\
rgg-400-1816             & \textbf{              15}& \textbf{              15}$^{\ast}$& & \textbf{              15}& \textbf{              15}$^{\ast}$& &               16 & &               17 &               17 &               17 &               17 \\
rgg-1000-2792            & \textbf{              44}& \textbf{              44}$^{\ast}$& & \textbf{              44}& \textbf{              44}$^{\ast}$& & \textbf{              44}$^{\ast}$& &               45 &               49 &               45 &               45 \\
harary17c3               & \textbf{               6}& \textbf{               6}$^{\ast}$& & \textbf{               6}& \textbf{               6}$^{\ast}$& & \textbf{               6}$^{\ast}$& & \textbf{               6}$^{\ast}$& \textbf{               6}$^{\ast}$& \textbf{               6}$^{\ast}$& \textbf{               6}$^{\ast}$\\
\end{tabular}\stopchange
\end{table}

\begin{table}[ht!]\change \caption{\change Results from experiments with ILP approaches and heuristics - synthetic single-source instances\stopchange} \label{tab:BT-single}
\centering\begin{tabular}{lcccccccccccc}
                         & \multicolumn{5}{c}{ILP approaches} && \multicolumn{6}{c}{Heuristics} \\ \cline{2-6} \cline{8-13}
                         & \multicolumn{2}{c}{de Sousa} && \multicolumn{2}{c}{Alg\ref{alg:ip}} && \multicolumn{1}{c}{ } && \multicolumn{4}{c}{Alg\ref{alg:match}} \\ \cline{2-3} \cline{5-6} \cline{10-13}
                Instance &               lb &               ub & &               lb &               ub & &             Lima & &              Ub1 &              Ub2 &              Ub3 &              Ub4 \\\hline
BT05\_RG050              & \textbf{               5}& \textbf{               5}$^{\ast}$& & \textbf{               5}& \textbf{               5}$^{\ast}$& & \textbf{               5}$^{\ast}$& &                6 &                6 & \textbf{               5}$^{\ast}$& \textbf{               5}$^{\ast}$\\
BT05\_RG075              & \textbf{               5}& \textbf{               5}$^{\ast}$& & \textbf{               5}& \textbf{               5}$^{\ast}$& & \textbf{               5}$^{\ast}$& &                6 &                6 & \textbf{               5}$^{\ast}$& \textbf{               5}$^{\ast}$\\
BT05\_RG100              & \textbf{               5}& \textbf{               5}$^{\ast}$& & \textbf{               5}& \textbf{               5}$^{\ast}$& & \textbf{               5}$^{\ast}$& &                6 & \textbf{               5}$^{\ast}$& \textbf{               5}$^{\ast}$& \textbf{               5}$^{\ast}$\\
BT05\_RG150              & \textbf{               5}& \textbf{               5}$^{\ast}$& & \textbf{               5}& \textbf{               5}$^{\ast}$& & \textbf{               5}$^{\ast}$& & \textbf{               5}$^{\ast}$& \textbf{               5}$^{\ast}$& \textbf{               5}$^{\ast}$& \textbf{               5}$^{\ast}$\\
\end{tabular}\stopchange
\end{table}

\begin{table}[ht!]\change \caption{\change Results from experiments with ILP approaches and heuristics - small world ($|V|=100$) single-source instances\stopchange} \label{tab:SW-100--single}
\centering\begin{tabular}{lcccccccccccc}
                         & \multicolumn{5}{c}{ILP approaches} && \multicolumn{6}{c}{Heuristics} \\ \cline{2-6} \cline{8-13}
                         & \multicolumn{2}{c}{de Sousa} && \multicolumn{2}{c}{Alg\ref{alg:ip}} && \multicolumn{1}{c}{ } && \multicolumn{4}{c}{Alg\ref{alg:match}} \\ \cline{2-3} \cline{5-6} \cline{10-13}
                Instance &               lb &               ub & &               lb &               ub & &             Lima & &              Ub1 &              Ub2 &              Ub3 &              Ub4 \\\hline
SW-100-4-0d1-trial1      & \textbf{               9}& \textbf{               9}$^{\ast}$& & \textbf{               9}& \textbf{               9}$^{\ast}$& & \textbf{               9}$^{\ast}$& &               10 &               10 &               10 & \textbf{               9}$^{\ast}$\\
SW-100-4-0d1-trial2      & \textbf{               8}& \textbf{               8}$^{\ast}$& & \textbf{               8}& \textbf{               8}$^{\ast}$& & \textbf{               8}$^{\ast}$& &                9 &                9 &                9 &                9 \\
SW-100-4-0d1-trial3      & \textbf{              10}& \textbf{              10}$^{\ast}$& & \textbf{              10}& \textbf{              10}$^{\ast}$& & \textbf{              10}$^{\ast}$& &               11 &               11 &               11 & \textbf{              10}$^{\ast}$\\
SW-100-4-0d2-trial1      & \textbf{               8}& \textbf{               8}$^{\ast}$& & \textbf{               8}& \textbf{               8}$^{\ast}$& & \textbf{               8}$^{\ast}$& &                9 &               10 &                9 & \textbf{               8}$^{\ast}$\\
SW-100-4-0d2-trial2      & \textbf{               8}& \textbf{               8}$^{\ast}$& & \textbf{               8}& \textbf{               8}$^{\ast}$& & \textbf{               8}$^{\ast}$& &                9 &                9 &                9 &                9 \\
\end{tabular}\stopchange
\end{table}

\begin{table}[ht!]\change \caption{\change Results from experiments with ILP approaches and heuristics - small world ($|V|=1000$) single-source instances\stopchange} \label{tab:SW-1000-single}
\centering\begin{tabular}{lcccccccccccc}
                         & \multicolumn{5}{c}{ILP approaches} && \multicolumn{6}{c}{Heuristics} \\ \cline{2-6} \cline{8-13}
                         & \multicolumn{2}{c}{de Sousa} && \multicolumn{2}{c}{Alg\ref{alg:ip}} && \multicolumn{1}{c}{ } && \multicolumn{4}{c}{Alg\ref{alg:match}} \\ \cline{2-3} \cline{5-6} \cline{10-13}
                Instance &               lb &               ub & &               lb &               ub & &             Lima & &              Ub1 &              Ub2 &              Ub3 &              Ub4 \\\hline
SW-1000-4-0d1-trial1     & \textbf{              15}& \textbf{              15}$^{\ast}$& & \textbf{              15}& \textbf{              15}$^{\ast}$& &               16 & &               17 &               18 &               17 &               18 \\
SW-1000-4-0d1-trial2     & \textbf{              16}& \textbf{              16}$^{\ast}$& & \textbf{              16}& \textbf{              16}$^{\ast}$& &               17 & &               18 &               18 &               19 &               17 \\
SW-1000-4-0d1-trial3     & \textbf{              16}& \textbf{              16}$^{\ast}$& & \textbf{              16}& \textbf{              16}$^{\ast}$& &               17 & &               18 &               18 &               18 &               18 \\
SW-1000-4-0d2-trial1     & \textbf{              12}& \textbf{              13}& & \textbf{              12}& \textbf{              13}& &               14 & &               14 &               14 &               15 &               14 \\
SW-1000-4-0d2-trial2     & \textbf{              12}& \textbf{              13}& & \textbf{              12}& \textbf{              13}& &               14 & &               14 &               15 &               14 &               14 \\
\end{tabular}\stopchange
\end{table}

\begin{table}[ht!]\change \caption{\change Results from experiments with ILP approaches and heuristics - miscellaneous double-source instances\stopchange} \label{tab:misc-double}
\centering\begin{tabular}{lcccccccccccc}
                         & \multicolumn{5}{c}{ILP approaches} && \multicolumn{6}{c}{Heuristics} \\ \cline{2-6} \cline{8-13}
                         & \multicolumn{2}{c}{de Sousa} && \multicolumn{2}{c}{Alg\ref{alg:ip}} && \multicolumn{1}{c}{ } && \multicolumn{4}{c}{Alg\ref{alg:match}} \\ \cline{2-3} \cline{5-6} \cline{10-13}
                Instance &               lb &               ub & &               lb &               ub & &             Lima & &              Ub1 &              Ub2 &              Ub3 &              Ub4 \\\hline
rgg-400-1220             & \textbf{              18}& \textbf{              18}$^{\ast}$& & \textbf{              18}& \textbf{              18}$^{\ast}$& & \textbf{              18}$^{\ast}$& &               19 &               21 &               19 &               19 \\
rgg-400-1779             & \textbf{              12}& \textbf{              12}$^{\ast}$& & \textbf{              12}& \textbf{              12}$^{\ast}$& & \textbf{              12}$^{\ast}$& &               15 &               14 &               14 &               13 \\
rgg-400-1816             & \textbf{              13}& \textbf{              13}$^{\ast}$& & \textbf{              13}& \textbf{              13}$^{\ast}$& &               14 & &               16 &               16 &               15 &               16 \\
rgg-600-1833             & \textbf{              23}& \textbf{              23}$^{\ast}$& & \textbf{              23}& \textbf{              23}$^{\ast}$& & \textbf{              23}$^{\ast}$& &               24 &               25 & \textbf{              23}$^{\ast}$&               24 \\
rgg-800-1034             & \textbf{             296}& \textbf{             296}$^{\ast}$& & \textbf{             296}& \textbf{             296}$^{\ast}$& & \textbf{             296}$^{\ast}$& &              298 &              299 &              298 & \textbf{             296}$^{\ast}$\\
\end{tabular}\stopchange
\end{table}

\begin{table}[ht!]\change \caption{\change Results from experiments with ILP approaches and heuristics - synthetic double-source instances\stopchange} \label{tab:BT-double}
\centering\begin{tabular}{lcccccccccccc}
                         & \multicolumn{5}{c}{ILP approaches} && \multicolumn{6}{c}{Heuristics} \\ \cline{2-6} \cline{8-13}
                         & \multicolumn{2}{c}{de Sousa} && \multicolumn{2}{c}{Alg\ref{alg:ip}} && \multicolumn{1}{c}{ } && \multicolumn{4}{c}{Alg\ref{alg:match}} \\ \cline{2-3} \cline{5-6} \cline{10-13}
                Instance &               lb &               ub & &               lb &               ub & &             Lima & &              Ub1 &              Ub2 &              Ub3 &              Ub4 \\\hline
BT05\_RG050              & \textbf{               4}& \textbf{               4}$^{\ast}$& & \textbf{               4}& \textbf{               4}$^{\ast}$& & \textbf{               4}$^{\ast}$& &                5 &                5 &                5 & \textbf{               4}$^{\ast}$\\
BT05\_RG075              & \textbf{               4}& \textbf{               4}$^{\ast}$& & \textbf{               4}& \textbf{               4}$^{\ast}$& & \textbf{               4}$^{\ast}$& &                5 &                5 & \textbf{               4}$^{\ast}$& \textbf{               4}$^{\ast}$\\
BT05\_RG100              & \textbf{               4}& \textbf{               4}$^{\ast}$& & \textbf{               4}& \textbf{               4}$^{\ast}$& & \textbf{               4}$^{\ast}$& &                5 & \textbf{               4}$^{\ast}$& \textbf{               4}$^{\ast}$& \textbf{               4}$^{\ast}$\\
BT05\_RG150              & \textbf{               4}& \textbf{               4}$^{\ast}$& & \textbf{               4}& \textbf{               4}$^{\ast}$& & \textbf{               4}$^{\ast}$& &                5 & \textbf{               4}$^{\ast}$& \textbf{               4}$^{\ast}$& \textbf{               4}$^{\ast}$\\
\end{tabular}\stopchange
\end{table}

\begin{table}[ht!]\change \caption{\change Results from experiments with ILP approaches and heuristics - small world ($|V|=100$) double-source instances\stopchange} \label{tab:SW-100--double}
\centering\begin{tabular}{lcccccccccccc}
                         & \multicolumn{5}{c}{ILP approaches} && \multicolumn{6}{c}{Heuristics} \\ \cline{2-6} \cline{8-13}
                         & \multicolumn{2}{c}{de Sousa} && \multicolumn{2}{c}{Alg\ref{alg:ip}} && \multicolumn{1}{c}{ } && \multicolumn{4}{c}{Alg\ref{alg:match}} \\ \cline{2-3} \cline{5-6} \cline{10-13}
                Instance &               lb &               ub & &               lb &               ub & &             Lima & &              Ub1 &              Ub2 &              Ub3 &              Ub4 \\\hline
SW-100-4-0d1-trial1      & \textbf{               8}& \textbf{               8}$^{\ast}$& & \textbf{               8}& \textbf{               8}$^{\ast}$& & \textbf{               8}$^{\ast}$& & \textbf{               8}$^{\ast}$& \textbf{               8}$^{\ast}$& \textbf{               8}$^{\ast}$& \textbf{               8}$^{\ast}$\\
SW-100-4-0d1-trial2      & \textbf{               7}& \textbf{               7}$^{\ast}$& & \textbf{               7}& \textbf{               7}$^{\ast}$& & \textbf{               7}$^{\ast}$& &                8 & \textbf{               7}$^{\ast}$& \textbf{               7}$^{\ast}$& \textbf{               7}$^{\ast}$\\
SW-100-4-0d1-trial3      & \textbf{               8}& \textbf{               8}$^{\ast}$& & \textbf{               8}& \textbf{               8}$^{\ast}$& & \textbf{               8}$^{\ast}$& &                9 &               10 &               10 &                9 \\
SW-100-4-0d2-trial1      & \textbf{               7}& \textbf{               7}$^{\ast}$& & \textbf{               7}& \textbf{               7}$^{\ast}$& & \textbf{               7}$^{\ast}$& &                8 &                8 &                8 &                8 \\
\end{tabular}\stopchange
\end{table}

\begin{table}[ht!]\change \caption{\change Results from experiments with ILP approaches and heuristics - small world ($|V|=1000$) double-source instances\stopchange} \label{tab:SW-1000-double}
\centering\begin{tabular}{lcccccccccccc}
                         & \multicolumn{5}{c}{ILP approaches} && \multicolumn{6}{c}{Heuristics} \\ \cline{2-6} \cline{8-13}
                         & \multicolumn{2}{c}{de Sousa} && \multicolumn{2}{c}{Alg\ref{alg:ip}} && \multicolumn{1}{c}{ } && \multicolumn{4}{c}{Alg\ref{alg:match}} \\ \cline{2-3} \cline{5-6} \cline{10-13}
                Instance &               lb &               ub & &               lb &               ub & &             Lima & &              Ub1 &              Ub2 &              Ub3 &              Ub4 \\\hline
SW-1000-4-0d1-trial1     & \textbf{              13}& \textbf{              13}$^{\ast}$& & \textbf{              13}& \textbf{              13}$^{\ast}$& &               15 & &               15 &               16 &               16 &               15 \\
SW-1000-4-0d1-trial2     & \textbf{              15}& \textbf{              15}$^{\ast}$& & \textbf{              15}& \textbf{              15}$^{\ast}$& & \textbf{              15}$^{\ast}$& &               16 &               16 &               17 &               16 \\
SW-1000-4-0d1-trial3     & \textbf{              14}& \textbf{              14}$^{\ast}$& & \textbf{              14}& \textbf{              14}$^{\ast}$& &               15 & &               16 &               16 &               16 &               16 \\
SW-1000-4-0d2-trial1     & \textbf{              11}& \textbf{              11}$^{\ast}$& & \textbf{              11}& \textbf{              11}$^{\ast}$& &               13 & &               13 &               13 &               13 &               13 \\
SW-1000-4-0d2-trial2     &               10 & \textbf{              11}$^{\ast}$& & \textbf{              11}& \textbf{              11}$^{\ast}$& &               12 & &               12 &               12 &               12 & \textbf{              11}$^{\ast}$\\
\end{tabular}\stopchange
\end{table}

\section{Running time} \label{ap:table-t}
\begin{table}[ht!]\change \caption{\change Running times (seconds) - miscellaneous graphs\stopchange } \label{tab:time-misc}
\centering\begin{tabular}{lrrrrrrcrrrrrr}
                     & \multicolumn{6}{c}{$|S|=1$} && \multicolumn{6}{c}{$|S|=2$} \\ \cline{2-7} \cline{9-14}
                     &\multicolumn{1}{c}{de}& & \multicolumn{4}{c}{Alg\ref{alg:match}}&&\multicolumn{1}{c}{de}& & \multicolumn{4}{c}{Alg\ref{alg:match}}\\ \cline{4-7} \cline{11-14}
Instance             & \multicolumn{1}{c}{Sousa} & \multicolumn{1}{c}{Alg\ref{alg:ip}} & Ub1& Ub2& Ub3& Ub4&& \multicolumn{1}{c}{Sousa} & \multicolumn{1}{c}{Alg\ref{alg:ip}} & Ub1& Ub2& Ub3& Ub4\\\hline
rgg-400-1220         &   93 &   76 &    0 &    0 &    1 &    1 &&   10 &    3 &    0 &    0 &    1 &    1 \\
rgg-400-1779         & 1505 &   18 &    0 &    0 &    1 &    1 && 1294 &    7 &    0 &    0 &    1 &    1 \\
rgg-400-1816         & 1373 &   10 &    0 &    0 &    1 &    1 &&   74 &    7 &    0 &    0 &    1 &    1 \\
rgg-600-1833         &     -- &     -- &     -- &     -- &     -- &     -- &&   39 &   35 &    0 &    1 &    1 &    1 \\
rgg-800-1034         &     -- &     -- &     -- &     -- &     -- &     -- && 1338 &   11 &    3 &    5 &    6 &    8 \\
\end{tabular}\stopchange
\end{table}

\begin{table}[ht!]\change \caption{\change Running times (seconds) - small world graphs\stopchange } \label{tab:time-SW}
\centering\begin{tabular}{lrrrrrrcrrrrrr}
                     & \multicolumn{6}{c}{$|S|=1$} && \multicolumn{6}{c}{$|S|=2$} \\ \cline{2-7} \cline{9-14}
                     &\multicolumn{1}{c}{de}& & \multicolumn{4}{c}{Alg\ref{alg:match}}&&\multicolumn{1}{c}{de}& & \multicolumn{4}{c}{Alg\ref{alg:match}}\\ \cline{4-7} \cline{11-14}
Instance             & \multicolumn{1}{c}{Sousa} & \multicolumn{1}{c}{Alg\ref{alg:ip}} & Ub1& Ub2& Ub3& Ub4&& \multicolumn{1}{c}{Sousa} & \multicolumn{1}{c}{Alg\ref{alg:ip}} & Ub1& Ub2& Ub3& Ub4\\\hline
SW-100-4-0d3-trial2  &   13 &    1 &    0 &    0 &    0 &    0 &&    6 &    0 &    0 &    0 &    0 &    0 \\
SW-100-5-0d3-trial2  &   33 &    3 &    0 &    0 &    0 &    0 &&    2 &    0 &    0 &    0 &    0 &    0 \\
SW-100-6-0d1-trial2  &   22 &    0 &    0 &    0 &    0 &    0 &&    1 &    1 &    0 &    0 &    0 &    0 \\
SW-1000-4-0d1-trial1 & 2448 &   69 &    0 &    1 &    1 &    1 &&  810 &  211 &    0 &    0 &    1 &    1 \\
SW-1000-4-0d1-trial2 &  291 &   59 &    0 &    0 &    1 &    1 &&   41 &   30 &    0 &    0 &    1 &    2 \\
\end{tabular}\stopchange
\end{table}

\end{appendices}

\end{document}